\documentclass[12pt]{amsart}
\usepackage{amsmath,amssymb,amsthm}

\DeclareMathOperator{\supp}{supp} 
\DeclareMathOperator{\diag}{diag}

\DeclareMathOperator{\dist}{d} 

\newcommand{\la}{\langle} \newcommand{\ra}{\rangle}

 \newcommand{\eps}{\varepsilon}
\newcommand{\ls}{\lesssim} \newcommand{\gs}{\gtrsim}
\newcommand{\R}{\mathbb{R}} \newcommand{\C}{\mathbb{C}}
\newcommand{\N}{\mathbb{N}} \newcommand{\Z}{\mathbb{Z}}

\newcommand{\al}{\alpha} \newcommand{\om}{\omega}
\newcommand{\lam}{\lambda} \newcommand{\ka}{\kappa}
\newcommand{\F}{\mathcal{F}}

\newtheorem{thm}{Theorem} \newtheorem{cor}[thm]{Corollary}
\newtheorem{pro}[thm]{Proposition} \newtheorem{lem}[thm]{Lemma}

\theoremstyle{remark}

\theoremstyle{definition}

\numberwithin{equation}{section} \numberwithin{thm}{section}

\begin{document}

\title[The cubic Dirac equation in $H^1(\mathbb{R}^3)$]{
The cubic
  Dirac equation:\\ Small initial data in $H^1(\mathbb{R}^3)$}

\author[I. Bejenaru]{Ioan Bejenaru} \address[I. Bejenaru]{Department
  of Mathematics, University of California, San Diego, La Jolla, CA
  92093-0112 USA} \email{ibejenaru@math.ucsd.edu}

\author[S. Herr]{Sebastian Herr} \address[S. Herr]{Fakult\"at f\"ur
  Mathematik, Universit\"at Bielefeld, Postfach 10 01 31, 33501
  Bielefeld, Germany} \email{herr@math.uni-bielefeld.de}

\thanks{The first author was supported in part by NSF grant DMS-1001676.
The second author acknowledges support from the German Research
Foundation, Collaborative Research Center 701.}

\begin{abstract}
  We establish global well-posedness and scattering
  for the cubic Dirac equation for small data in the critical space
  $H^1(\R^3)$. The main ingredient is obtaining a sharp
  end-point Strichartz estimate for the Klein-Gordon equation. In a classical sense
  this fails and it is related to the failure of the endpoint Strichartz estimate for the
  wave equation in space dimension three. In this
  paper, systems of coordinate frames are constructed in which
  endpoint Strichartz estimates are recovered and energy estimates are
  established.
\end{abstract}


\maketitle

\section{Introduction and main results}\label{sect:intro}
For $m >0$, consider the scalar homogeneous Klein-Gordon equation
\begin{equation} \label{KG} \Box u(t,x) + m^2 u(t,x) =0, \qquad (t,x)
  \in \R \times \R^n.
\end{equation}
A fundamental problem is the validity of Strichartz estimates for
solutions of this equation. In the low frequency regime, the
dispersive properties of the Klein-Gordon equation are similar to the
Schr\"odinger equation, while in the high frequency regime they are
similar to the wave equation (this will be detailed later in the
paper). This hints at the range of available Strichartz estimates for
\eqref{KG}.

In dimensions $n \geq 4$, it is known that all the Strichartz
estimates including the end-point hold true both for the Schr\"odinger
and the wave equation \cite{KT98}.  Therefore all the Strichartz
estimates including the end-point hold true for the Klein-Gordon
equation as well.

A major problem arises in dimension $n=3$ since the endpoint
Strichartz estimate $L^2_t L^\infty_x$ fails for the wave equation due
to the slow dispersion of type $t^{-1}$. On the positive side, the
end-point Strichartz estimate $L^2_t L^6_x$ holds true for the
Schr\"odinger equation. Therefore, the problem one encounters for the
Klein-Gordon equation is in the high frequency regime only.

Strichartz estimates lead to well-posedness results for various
nonlinear equations.  The endpoint Strichartz estimate plays a crucial
role in certain critical problems.  The application discussed in this
paper, the cubic Dirac equation, is such an example. In fact this
equation motivated our research in the direction of obtaining a
replacement for the false endpoint Strichartz estimate for \eqref{KG}.

In a future work we will address the same problem in two dimensions
where the $L^2_t L^\infty_x$ estimate fails for the Schr\"odinger
equation and it is not even the correct end-point for the wave
equation.

Throughout the rest of this paper the physical dimension is set to
$n=3$ and the mass is fixed to $m=1$ in \eqref{KG}. By rescaling,
estimates for any other $m \ne 0$ can be obtained. It is well-known
that in the case of the wave equation, \[\Box u=0, u(0,x)=f_0(x),
u_t(0,x)=f_1(x),\] the end-point Strichartz estimate
\begin{equation} \label{EPS} \| u \|_{L^2_t L^\infty_x} \ls \| \nabla
  f_0 \|_{L^2} + \| f_1 \|_{L^2}
\end{equation}
does not hold true, see \cite{M-S}. In fact it fails for any $P(D) u$
where $P(D)$ is a Fourier multiplier whose symbol lies in
$C_0^\infty$, vanishes near the origin and it is not identically zero,
see \cite{Tao-bil}. In particular it fails for $P_k u$, where $P_k$ is
the standard Fourier multiplier localizing at frequency $|\xi| \approx
2^k$, see Subsection \ref{subsect:not}. As a consequence the estimate
\eqref{EPS} cannot hold true for \eqref{KG} either.  To be more
precise, the estimate \eqref{EPS} for $P_k u$ with a bound independent
of $k$ cannot be true. This obstruction comes as $k \rightarrow
\infty$ where the symbol of the Klein-Gordon equation is essentially
the same as the one for the wave equation.

An important observation needs to be made here. While for the wave
equation \eqref{EPS} is false regardless on how much regularity is
added to the right hand side, that is to $f_0,f_1$, some extra
regularity fixes the estimate for the Klein-Gordon equation.  To be
more precise, if
\[
(\Box+1)u=0, u(0,x)=f_0(x), u_t(0,x)=f_1(x),
\]
the end-point Strichartz estimate
\begin{equation} \label{EPSK} \| P_k u \|_{L^2_t L^\infty_x}
  \ls_\epsilon 2^{(1+\epsilon) k} \| P_k f_0 \|_{L^2} + 2^{\epsilon k}
  \| P_k f_1 \|_{L^2}, \qquad k \geq 0,
\end{equation}
holds true for any $\epsilon > 0$, see \cite{MNO}. But this fails to
be true for $\epsilon=0$!

Our goal in this paper is to provide a lucrative replacement for
\eqref{EPSK} in the case $\epsilon=0$ and for its inhomogeneous
counterpart. This will done in adapted frames in Section \ref{endpnt},
see Theorem \ref{thm:estr}. In applications to nonlinear problems, the
end-point Strichartz estimate is used in conjunction with the energy
estimate $L^\infty_t L^2_x$ to generate the bilinear $L^2_{t,x}$
estimate
\[
\| u \cdot v \|_{L^2_{t,x}} \leq \| u \|_{L^2_t L^\infty_x} \| u
\|_{L^\infty_t L^2_x}.
\]
Since the $L^2 L^\infty$ estimate is generated in adapted frames, one
has to derive energy estimates in similar frames in order to recoup
the above $L^2_{t,x}$ bilinear estimate. We will provide this type of
energy estimates in Subsection \ref{Energy}. In fact, the combination
of the energy and the Strichartz estimate to a uniform $L^2$ estimate
is only possible by using a null structure, see Subsection
\ref{subsect:nullst}.

The use of adapted frames to generate a replacement for the missing
$L^2_t L^\infty_x$ end-point Strichartz estimate was initiated by
Tataru \cite{tat} in the context of the Wave Map problem. Another
context in which such estimates were derived was the Schr\"odinger Map
problem, see \cite{bikt}. Our work is closer in spirit to the work of
Tataru \cite{tat}, although the geometry of the characteristic surface
for the Klein-Gordon equation requires a more involved
construction.

As an application, we study the cubic Dirac equation which we describe
below.  For $M> 0$, the cubic Dirac equation for the spinor field
$\psi: \R^4 \to \C^4$ is given by
\begin{equation} \label{eq:dirac} (-i \gamma^\mu \partial_\mu + M )
  \psi= \la \gamma^0 \psi, \psi \ra \psi,
\end{equation}
where we use the summation convention. Here, $\gamma^\mu\in
\C^{4\times 4}$ are the Dirac matrices given by
\[
\gamma^0= \left( \begin{array}{cc} I_2 & 0 \\ 0 & -I_2 \end{array}
\right) , \qquad \gamma^j=\left( \begin{array}{cc} 0 & \sigma^j \\
    -\sigma^j & 0 \end{array} \right)
\]
where
\[
\sigma^1= \left( \begin{array}{cc} 0 & 1 \\ 1 & 0 \end{array} \right)
, \qquad \sigma^2=\left( \begin{array}{cc} 0 & -i \\ i & 0 \end{array}
\right) , \qquad \sigma^3=\left( \begin{array}{cc} 1 & 0 \\ 0 &
    -1 \end{array} \right)
\]
are the Pauli matrices. The $\langle \cdot, \cdot \rangle$ is the
standard scalar product on $\C^4$, hence $\la \gamma^0 \psi, \psi \ra=
|\psi_1|^2+|\psi_2|^2 - |\psi_3|^2-|\psi_4|^2 \in \R$. It then follows
that $\la \gamma^0 \psi, \psi \ra$ equals its conjugate which is
written as $\bar \psi \psi= \psi^\dag \gamma^0 \psi$, where $\bar
\psi=\psi^\dag \gamma^0$ and $\psi^\dag$ is the conjugate transpose of
$\psi$. The conclusion is that $\la \gamma^0 \psi, \psi \ra= \psi^\dag
\gamma^0 \psi$ and we made this point so as to avoid confusion between
the two apparently different ways the nonlinear term appears in
literature.

The matrices $\gamma^\mu$ satisfy the following properties
\[
\gamma^\alpha \gamma^\beta + \gamma^\beta \gamma^\alpha = 2 g^{\alpha
  \beta} I_4, \qquad (g^{\alpha \beta})= \diag(1,-1,-1,-1).
\]
The physical background for this equation is provided in
\cite{flr,soler}.  Existence and stability of bound state solutions of
\eqref{eq:dirac} has been investigated in \cite{sv,cv,M88}.

Using scaling arguments, it turns out that the problem becomes
critical in $H^1(\R^3)$. Local well-posedness was obtained in
$H^s(\R^3)$, $s > 1$ (subcritical range) in \cite{ev}. Global
well-posedness and scattering was proved in \cite{MNO} for small
initial data in $H^s(\R^3)$, $s >1$ as well as for small initial data
in $H^1(\R^3)$ with some regularity in the angular variable in \cite{MNNO}.

The main idea in the above mentioned papers is as follows. The linear
part of the Dirac equation is closely related to a half-Klein-Gordon
equation.  In the subcritical case one can make use of the
\eqref{EPSK} with $\epsilon >0$, while in the critical case certain
spherically averaged versions \eqref{EPSK} with $\epsilon=0$ hold
true, see \cite{MNNO,KaOz}, which is similar to the Schr\"odinger case
\cite{Tao2} in dimension $n=2$.
  
Both of the above strategies reach their limitations when one
considers the \eqref{eq:dirac} with small but general data in
$H^1(\R^3)$, cp. \cite[p.~181, l.~1-5]{MNO}. Using our strategy to fix
\eqref{EPSK} in the case $\epsilon=0$ and the null structure exhibited
by the nonlinearity we are able to prove the following result in the
critical space:

\begin{thm}\label{thm:main}
  The initial value problem associated to the cubic Dirac equation
  \eqref{eq:dirac} is globally well-posed for small initial data in
  $H^1(\R^3)$. Moreover, small solutions scatter to free solutions for
  $t\to \pm \infty$.
\end{thm}

In addition, the result includes persistence of initial regularity,
i.e.\ if $\psi(0) \in H^\sigma(\R^3)$ for some $\sigma\geq 1$, the
solution $t\mapsto \psi(t)$ is a continuous curve in $H^\sigma(\R^3)$,
which in the case $\sigma>1$ is already known from the previous work
\cite{MNO}.

In a future work we intend to address the initial value problem for the cubic Dirac equation
in the critical space in space
dimension $n=2$.

For a subcritical result for the cubic Dirac equation in space
dimension $n=2$, see \cite{Pe13}, for results in space dimension
$n=1$, see \cite{MNT10,C11}. Concerning nonlinear Klein-Gordon
equations we refer the reader to \cite{DeFa,Ko,Kl,Sh}.

The plan for the paper is as follows. In the following subsection we
introduce the main notation which will be used throughout the rest of
the paper. In Section \ref{sect:le} we derive the major linear
estimates of the paper: the end-point $L^2 L^\infty$ in frames in
subsection \ref{endpnt} and the energy estimates in similar frames in
subsection \ref{Energy}. The proofs of some of the decay estimates are
postponed to Appendix \ref{app:decay}.  In Section \ref{sect:setupD}
we prepare the setup for the Dirac equation and unveil the null
condition present in the nonlinearity. In Section \ref{fspaces} we
introduce our function spaces, in Section \ref{sect:bil-est} we prove
useful bilinear estimates, which are applied in Section
\ref{sect:dirac} to prove the main result concerning the cubic Dirac
equation.

\subsection{Notation}\label{subsect:not}
We define $A\prec B$ by $A\leq B-c$ for some absolute constant
$c>0$. Also, we define $A\ll B$ to be $A\leq d B$ for some absolute
small constant $0<d<1$. Similarly, we define $A\ls B$ to be $A\leq e
B$ for some absolute constant $e>0$, and $A \approx B$ iff $A\ls B \ls
A$.

Similar to \cite{MNNO}, we set $\la \xi
\ra_k:=(2^{-2k}+|\xi|^2)^{\frac12}$ for $k \in \Z$ and $\xi \in \R^n$,
and we also write $\la\xi\ra:=\la\xi\ra_0$. We note that $\la \xi
\ra_k$ coincides with the euclidean norm of the vector
$(\xi,2^{-k})\in \R^{n+1}$. Since the euclidean norm is a smooth
function, homogeneous of degree $1$, on $\R^{n+1}\setminus \{0\}$, we
conclude that for all multi-indices $\beta \in \N_0^n$ there are
$c_{\beta,n}>0$, such that
\begin{equation}\label{eq:jap}
  \forall k \in \Z, \, \xi\in \R^n: \quad  |\partial^\beta_\xi \la \xi \ra_k|\leq c_{\beta,n} \la \xi\ra_k^{1-|\beta|}.
\end{equation}

Thoughout the paper, let $\rho\in C^\infty_c(-2,2)$ be a fixed smooth,
even, cutoff satisfying $\rho(s)=1$ for $|s|\leq 1$ and $0\leq
\rho\leq 1$. For $k \in \Z$ we define $\chi_k:\R^n \to \R$,
$\chi_k(y):=\rho(2^{-k}|y|)-\rho(2^{-k+1}|y|)$, such that
$A_k:=\supp(\chi_k)\subset \{y \in \R^n \colon 2^{k-1}\leq |y|\leq
2^{k+1}\}$. Let $\tilde{\chi}_k=\chi_{k-1}+\chi_k+\chi_{k+1}$ and
$\tilde{A}_k:=\supp(\tilde{\chi}_k)$.

We denote by $P_k = \chi_k(D)$ and $\tilde{P}_k =
\tilde{\chi}_k(D)$. Note that $P_k\tilde{P}_k=\tilde{P}_kP_k=P_k$.
Further, we define $\chi_{\leq k}=\sum_{l=-\infty}^k \chi_l$, $\chi_{>
  k}=1-\chi_{\leq k}$ as well as the corresponding operators $P_{\leq
  k}=\chi_{\leq k}(D)$ and $P_{> k}=\chi_{> k}(D)$.

We denote by $\mathcal{K}_l$ a collection of spherical caps of
diameter $2^{-l}$ which provide a symmetric and finitely overlapping
cover of the unit sphere $\mathbb{S}^2$. Let $\om(\ka)$ to be the
"center" of $\ka$ and let $\Gamma_{\ka}$ be the cone generated by
$\ka$ and the origin, in particular $\Gamma_{\ka}\cap \mathbb{S}^2
=\ka$.

For $M_1,M_2\subset \R^n$ we set \[\dist(M_1,M_2)=\inf\{|x-y|\colon
x\in M_1, y\in M_2\}.\]

Further, let $\eta_{\ka}$ be smooth partition of unity subordinate to
the covering of $\R^3\setminus\{0\}$ with the cones $\Gamma_{\ka}$,
such that each $\eta_\ka$ is supported in $2 \Gamma_{\ka}$ and is
homogeneous of degree zero and satisfies
\[
|\partial_\xi^\beta\eta_\kappa(\xi)|\leq C_\beta
2^{l|\beta|}|\xi|^{-\beta}, \quad
|(\omega(\ka)\cdot\nabla)^N\eta_\kappa(\xi)|\leq C_N |\xi|^{-N},
\]
as described in detail in \cite[Chapt. IX, \S 4.4 and formula
(66)]{St}. Let $\tilde{\eta}_{\ka}$ with similar properties but
slightly bigger support, such that $\tilde{\eta}_{\ka}\eta_{\ka}=1$.
We define $P_{\kappa}=\eta_{\ka}(D)$,
$\tilde{P}_{\kappa}=\tilde{\eta}_{\ka}(D)$.  With $P_{k,\ka}:=
\eta_{\ka}(D) \chi_k(D)$ and $\tilde{P}_{k,\ka}:=
\tilde{\eta}_{\ka}(D) \tilde{\chi}_k(D)$, we obtain the angular
decomposition
\[
P_k = \sum_{\ka \in \mathcal{K}_l} P_{k,\ka}
\]
and $P_{k,\ka}\tilde{P}_{k,\ka}=\tilde{P}_{k,\ka}P_{k,\ka}=P_{k,\ka}$.
We further define $A_{k,\ka}=\supp(\eta_{\ka} \chi_k)$ and
$\tilde{A}_{k,\ka}=\supp(\tilde{\eta}_{\ka} \tilde{\chi}_k)$.

We define $\widehat{Q^\pm_m u}(\tau,\xi)=\chi_m(\tau\mp \la
\xi\ra)\widehat{u}(\tau,\xi)$, and $\widehat{Q^\pm_{\leq m}
  u}(\tau,\xi)=\chi_{\leq m}(\tau\mp \la
\xi\ra)\widehat{u}(\tau,\xi)$. We also define
$\tilde{Q}^\pm_{m}=Q^\pm_{m-1}+Q^\pm_{m}+Q^\pm_{m+1}$.  We set
$B^\pm_{k,m}$ to be the Fourier support of $Q^\pm_m$, and
$\tilde{B}^\pm_{k,m}$ to be the Fourier support of $\tilde{Q}^\pm_m$.
Additionally, we define $Q^\pm_{\prec m} =\sum_{l=-\infty}^{m-c}
Q^\pm_l$ for a large integer $c>0$, and $Q^\pm_{\succeq m}=I-Q^\pm_{\prec m}$. Given $k \in \Z$, and $\kappa \in
\mathcal{K}_l$ for some $l \in \N$ we set $B^\pm_{k,\kappa}$ to be the
Fourier-support of $Q^\pm_{\prec k-2l} P_{k,\kappa}$. Similarly we
define $\tilde B^\pm_{k,\kappa}$.

Given an angle $\omega$ and a parameter $\lambda$ we define the
directions $\Theta_{\lambda,\omega}=\frac1{\sqrt{1+\lambda^2}}(\lambda
,\omega)$, $\Theta^\perp_{\lambda,\omega} =
\frac1{\sqrt{1+\lambda^{2}}}(-1,\lambda\omega)$ and the associated
orthogonal coordinates $(t_\Theta,x^1_{\Theta}, x'_{\Theta})$
\[
t_{\lambda,\omega}= (t,x) \cdot \Theta_{\lambda,\omega}, \quad
x^1_{\lambda,\omega}= (t,x) \cdot \Theta^\perp_{\lambda,\omega}.
\]

If $\lambda=1$ we obtain the characteristic directions (null
co-ordinates) as in \cite[p.\ 42]{tat} and \cite[p.\
476]{tao}. However, our analysis requires more flexibility in the
choice of the frames with respect to which the estimates are
available. With $\om(\ka)$ defined above and $\lambda(k)=(1+2^{-2k})^{-\frac12}$
let $(t^\pm_{k,\ka},x^\pm_{k,\ka})=(t_{ \pm
  \lambda(k),\om(\ka)},x_{\pm \lambda(k),\om(\ka)})$.

For $1\leq p,q\leq \infty$ we use the spaces $L^p_t L^q_x$ of all
equivalence classes of measurable (weak-$*$-measurable if $q=\infty$)
functions $f:\R\to L^q(\R^3)$ such that the norm
\[
\|f\|_{L^pL^q}=\|t\mapsto \|f(t)\|_{L^q(\R^3)}\|_{L^p(\R)}
\]
is finite.
\section{Linear estimates}\label{sect:le}
The decay rates of solutions to the linear wave equation and
Klein-Gordon equation have been analyzed e.g. in
\cite{vW,S76,Str77,MS,Si,Kl2,GV85,B85,MSW80}, see also the references
therein. From the harmonic analysis point of view, the decay is
determined by the curvature properties of the characteristic sets. In
particular, we refer the reader to \cite[Section 2.5]{NaSch} for a
detailed discussion of decay and Strichartz estimates in the context
of the Klein-Gordon equation.

For convenience, we set $m=1$ in the Klein-Gordon equation \eqref{KG}.
By rescaling our analysis extends to \eqref{KG} with any $m \neq
0$. With $m=1$, the solution is given by
\begin{equation}\label{eq:kg-prop}
  u(t) = \frac12 (e^{it \la D \ra} + e^{-it \la D \ra}) u_0 + \frac1{2i}
  (e^{it \la D \ra} - e^{-it \la D \ra}) \frac{u_1}{\la D \ra}.
\end{equation}
where $\la D \ra$ is the Fourier multiplier with symbol $\la \xi \ra$.
It then becomes clear that the key operator to study is $e^{\pm it \la
  D \ra}$. To keep things simple, we work all estimates for the $+$
sign choice, that is for $e^{it \la D \ra}$. The estimates for $e^{-
  it \la D \ra}$ are obtained in a similar way by simply reversing
time in the estimates for $e^{it \la D \ra}$.

\subsection{End-point $L^2L^\infty$ type Strichartz estimate.}\label{endpnt}
Our main result in this section provides the end-point Strichartz
estimates available for functions localized in frequency.
\begin{thm}\label{thm:estr}
{\rm i)} For all $k \ls 1$ and $f \in L^2(\R^3)$
    satisfying $\supp(\widehat{f})\subset \tilde{A}_k$,
    \begin{equation} \label{eq:lf} \| e^{it \la D \ra} f \|_{L^2_t
        L^\infty_x} \ls 2^{\frac{k}2} \| f \|_{L^2}
    \end{equation}
{\rm ii)}  For all $k \gs 1$, $\ka \in \mathcal{K}_{k}$
    and $f \in L^2(\R^3)$ satisfying $\supp(\widehat{f}) \subset
    \tilde{A}_{k,\ka}$,
    \begin{equation} \label{eq:estr} 2^{-k} \| e^{it \la D \ra} f
      \|_{L^2_t L^\infty_x} + \| e^{it \la D \ra} f
      \|_{L^2_{t_{k,\ka}} L^\infty_{x_{k,\ka}}} \ls \| f \|_{L^2}.
    \end{equation}
{\rm iii)}  For all $k \gs 1$, $1 \leq l \leq k$, $ \ka_1
    \in \mathcal{K}_l$ and $f \in L^2(\R^3)$ satisfying
    $\supp(\widehat{f}) \subset \tilde A_{k, \ka_1}$,
    \begin{equation} \label{MStr} \sum_{\ka \in \mathcal{K}_k} \|e^{it
        \la D \ra} \tilde P_{\ka}
      f\|_{L^2_{t_{k,\ka}}L^\infty_{x_{k,\ka}}} \ls 2^{k-l}
      \|f\|_{L^2}.
    \end{equation}
\end{thm}

Part i) claims that for the low frequencies the end-point
Strichartz estimates holds in a standard fashion.  Given that in that
regime the evolution is Schr\"odinger-like, the correct end-point
would be $L^2_t L^6_x$ from which the estimate \eqref{eq:lf} can be
obtained using the Sobolev embedding theorem.

In \eqref{eq:estr} we reveal the main Strichartz estimates in high
frequencies. If we localize $\hat f$ in the angular variable at scale
$2^{-k}$ we obtain two Strichartz estimates. The standard one $L^2_t
L^\infty_x$ is obtained without any logarithmic loss, which would be
the case in the absence of angular localization. The Strichartz
estimate in characteristic coordinates is better adapted to the
direction in which the waves propagate and hence it comes with a much
better prefactor. The other key advantage that the Strichartz estimate
in characteristic coordinates has is revealed in \eqref{MStr} where at
each scale (larger than $2^{-k}$) of angular localization we obtain
the $l^1$ structure on pieces measured in $L^2L^\infty$ in
characteristic coordinates.  In particular when no angular
localization is present ($l=0$) one obtains a replacement of the
missing end-point $L^2_tL^\infty_x$ with the correct factor of
$2^k$. The use of so many frames to capture the $L^2L^\infty$ estimate
will require more flexibility in the corresponding energy estimates.

The rest of this subsection is devoted to the proof of Theorem
\ref{thm:estr}.  Define the kernel
\begin{equation}\label{K_kdef}
  K_k(t,x)=\int_{\R^3}e^{ix\cdot\xi}e^{it\la \xi \ra}\tilde{\chi}_k^2(|\xi|)\,d\xi.
\end{equation}
We identify $L^2_t L^\infty_x$ as the dual of $L^2_t L^1_x$, see
\cite[Theorem 8.20.3]{E65} and the definitions in Subsection
\ref{subsect:not}. Through the usual $TT^*$ argument, see
e.g. \cite[Lemma 2.1]{gv}, the low frequency case \eqref{eq:lf} 
 follows from
\begin{equation} \label{eq:destr} \| K_k * g \|_{L^2_t L^\infty_x} \ls
  2^{k} \| g \|_{L^2_t L^1_x}.
\end{equation}

The following result can be found in \cite[Corollaries 2.36 and
2.38]{NaSch}, it can be traced back to \cite{GV85,B85,MSW80}.
   
 \begin{lem}\label{lem:decay}
  {\rm i)} For all $k \in \Z$, $k\ls 1$, we have
     \begin{equation}\label{eq:smallk}
       |K_k(t,x)| \ls 2^{3k} (1+2^{2k}|(t,x)|)^{-\frac{3}{2}}.
     \end{equation}

{\rm ii)} For all $k\in \Z, k \gs 1$ we have
     \begin{equation}\label{eq:bigk}
       |K_k(t,x)| \ls
       2^{3k} (1+2^{k} |(t,x)|)^{-1} \min(1,(1+2^k|(t,x)|)^{-\frac12}2^k))
     \end{equation}
 \end{lem}
 Estimate \eqref{eq:smallk} easily follows from the classical result
 on Fourier transforms of surface carried measures \cite[p. 348,
 Theorem 1]{St}.  The idea behind estimate \eqref{eq:bigk} is the
 following: After rescaling to unit frequency size, $K_k$ essentially
 is the (inverse) Fourier transform of an approximately cone-like
 surface with $2$ principal curvatures which are uniformly bounded
 from below, cp. \cite{Li} or \cite[p.\ 361]{St}, which implies
 \eqref{eq:bigk} for $|(t,x)|\leq 2^k$. By taking into account that
 the surface actually has $n$ non-vanishing principal curvatures, one
 of which is of size $2^{-2k}$, cp. \cite[p.\ 360]{St} or
 \cite[Section 7]{H1} one obtains \eqref{eq:bigk} for $|(t,x)|\geq
 2^k$. For convenience of the reader, we provide a proof in Appendix \ref{app:decay}.

 Using the above Lemma, we obtain $\|K_k\|_{L^1_t
   L^\infty_x} \ls 2^k$ from which \eqref{eq:destr} and therefore
 \eqref{eq:lf} follows.  We are now left with completing the most
 interesting part of the argument, namely the proof of
 \eqref{eq:estr}. Through the $TT^*$ argument, the estimate
 \eqref{eq:estr} is reduced to the following
 \begin{equation*} \| K_{k,\ka} * g \|_{L^2_{t} L^\infty_{x}} \ls
   2^{2k} \| g \|_{L^2_{t} L^1_{x}} , \quad\| K_{k,\ka} * g
   \|_{L^2_{t_{k,\kappa}} L^\infty_{x_{k,\kappa}}} \ls \| g
   \|_{L^2_{t_{k,\kappa}} L^1_{x_{k,\kappa}}}
 \end{equation*}
 for $\ka \in \mathcal{K}_{k}$, where
 \begin{equation}\label{A_kdef}
   K_{k,\ka} (t,x)=\int_{\R^3}e^{ix\cdot\xi}e^{ it\la \xi \ra}\tilde{\chi}_k^2(|\xi|)\tilde{\eta}_\kappa(\xi)\,d\xi.
 \end{equation}
 Again by Young's inequality, this reduces to showing that
 \begin{equation}\label{eq:l1linftykappa}
   2^{-2k} \| K_{k,\ka} \|_{L^1_t L^\infty_x} 
   +  \| K_{k,\ka} \|_{L^1_{t_{k,\kappa}} L^\infty_{x_{k,\kappa}}} \ls 1.
 \end{equation}

 This estimate follows from the Proposition below.

 \begin{pro}\label{pro:ang} For all $k \in \Z, k \gs 1$, $\ka \in
   \mathcal{K}_{k}$ and all $(t,x)$
   \begin{equation}\label{eq:ang1}
     |K_{k,\ka}(t,x)|\ls2^{k}(1+ 2^{-k} |(t,x)|)^{-\frac32}.
   \end{equation}
   In addition, for $N=1,2$, we have the following:
   \begin{equation}\label{eq:ang2}
     |K_{k,\ka}(t,x)|\ls_N 2^{k}(1+ 2^k |t_{k,\ka}|)^{-N}, \quad \text{if } |t_{k,\ka}| \gg 2^{-2k} |(t,x)|.
   \end{equation}
  
 \end{pro}

We remark that \eqref{eq:ang2} holds with any $N \in \N$, but as stated it suffices for our purposes.

 Before turning to the proof of this Proposition, we show how
 \eqref{eq:l1linftykappa} follows from the statements above. The first
 part of \eqref{eq:l1linftykappa} is straightforward:
 \begin{align*}
   \| K_{k,\ka} \|_{L^1_t L^\infty_x} & \ls \int_{-2^{k}}^{2^k} \|
   K_{k,\ka} \|_{L^\infty_x} dt
   + \int^{-2^{k}}_{-\infty} \|K_{k,\ka} \|_{L^\infty_x} dt + \int_{2^k}^\infty \|K_{k,\ka} \|_{L^\infty_x} dt \\
   & \ls 2^{2k} + \int_{2^k}^\infty 2^{\frac{5k}2} t^{-\frac32} dt \ls
   2^{2k}.
 \end{align*}
 For the second part of \eqref{eq:l1linftykappa}, we want to
 understand $\| K_{k,\ka} (t_{k,\ka}, \cdot)
 \|_{L^\infty_{x_{k,\ka}}}$ for some fixed $t_{k,\ka}$ such that
 $|t_{k,\ka}| \approx 2^{j}$ with $j \geq -k$.  If the point
 $(t_{k,\ka},x_{k,\ka})$ belongs to the region $|t_{k,\ka}| \gg
 2^{-2k} |(t,x)|$, then we have the bound $|K_k(t,x)| \ls 2^k
 (2^{k+j})^{-2}$, while if it belongs to the region $|t_{k,\ka}| \ls
 2^{-2k} |(t,x)|$ then we have the bound $|K_k(t,x)| \ls 2^k
 (2^{-k}|(t,x)|)^{-\frac32} \ls 2^k (2^{k+j})^{-\frac32}$. The
 conclusion is that if $|t_{k,\ka}| \approx 2^j$ with $j \geq -k$ then
 \[
 \| K_{k,\ka} (t_{k,\ka}, \cdot) \|_{L^\infty_{x_{k,\ka}}} \ls 2^{k}
 2^{-\frac32 (k+j)}.
 \]
 From this we estimate
 \[
 \begin{split}
   \| K_{k,\ka} \|_{L^1_{t_{k,\kappa}} L^\infty_{x_{k,\kappa}}} &
   \ls \int_0^{2^{-k}} 2^k d t_{k,\ka}
   + \sum_{j =-k}^\infty \int_{2^j}^{2^{j+1}}  \| K_{k,\ka} (t_{k,\ka}, \cdot ) \|_{L^\infty_{x_{k,\ka}}} d t_{k,\ka} \\
   & \ls 1+ \sum_{j =-k}^\infty 2^{k+j} 2^{-\frac32 (k+j)} \ls 1
 \end{split}
 \]
 and this finishes the argument for the second part of
 \eqref{eq:l1linftykappa}. With this, the proof of \eqref{eq:estr} is
 complete.

 \begin{proof}[Proof of Proposition \ref{pro:ang}] We begin with the
   proof of \eqref{eq:ang1}. If $|(t,x)| \ls 2^k$ then the
   statement follows directly by using that size of the support of the
   integration has volume $\approx 2^k$. If $|(t,x)| \gs 2^k$, then
   the estimate follows from \eqref{eq:bigk} and Young's inequality.

   It remains to provide a proof of \eqref{eq:ang2}. For
   compactness of notation, we write $\lambda=\lambda(k)$,
   $\omega=\omega(\ka)$. By rescaling it suffices to consider
   \[
   B_{k,\ka}(s,y):=\int_{\R^3}e^{iy\cdot \xi + is\la \xi
     \ra_k}\tilde{\chi}_1^2(|\xi|)\tilde{\eta}_\ka(\xi)d\xi
   \]
   and establish, for $N=1,2$
   \begin{equation} \label{Best} |B_{k,\ka}(s,y)| \ls_N
     2^{-2k}(1+|s_{\lambda,\omega}|)^{-N}, \qquad |s_{\lambda,\omega}|
     \gg 2^{-2k} |(s,y)|.
   \end{equation} 
   If $|s_{\lambda,\omega}| \ls 1$, the estimate follows from the fact
   that the support of the integration has volume $\approx
   2^{-2k}$. For the rest of the argument we work under the hypothesis
   $|s_{\lambda,\omega}| > 1$.
  
   We write $(s,y)=\beta (r,z)$ with $\beta=|(s,y)|$ and the integral
   above becomes
   \[
   C_{k,\ka}(\beta,r,z)=\int_{\R^3}e^{i\beta
     \phi(r,z,\xi)}\tilde{\chi}_1^2(|\xi|)\tilde{\eta}_\ka(\xi)d\xi
   \]
   with $\phi(r,z,\xi)= z \cdot \xi + r \la \xi \ra_k$.  The phase
   function satisfies $ \partial_{\xi_j} \phi (r,z,\xi) = z_j + r
   \frac{\xi_j}{\la \xi \ra_k}.$
   Define $\partial_\omega =\omega\cdot \nabla_\xi$, $
   d_\phi :=\frac{1}{\partial_\omega 
     \phi}\partial_\omega $ and $
   d_\phi^\ast :=-\partial_\omega \Big(\frac{\cdot}{\partial_\omega
     \phi}\Big)$. Integrating by parts, we compute
   \begin{equation}\label{eq:ip}
     \begin{split}
       \int_{\R^3}e^{i\beta
         \phi(r,z,\xi)}\tilde{\chi}_1^2(|\xi|)\tilde{\eta}_\ka(\xi)d\xi
       & =
       \int_{\R^3} \frac1{(i \beta)^N} d_\phi^N (e^{i\beta \phi(r,z,\xi)}) 
        \tilde{\chi}_1^2(|\xi|)\tilde{\eta}_\ka(\xi)d\xi \\
       & = (i \beta)^{-N} \int_{\R^3} e^{i\beta
         \phi(r,z,\xi)} (d^*_\phi)^N (\tilde{\chi}_1^2(|\xi|)\tilde{\eta}_\ka(\xi)) d\xi
     \end{split}
   \end{equation}
   For $\zeta(\xi)=\tilde{\chi}_1^2(|\xi|)\tilde{\eta}_\ka(\xi)$ we
   claim the bounds
   \begin{equation}\label{eq:diffop}
     |(d^*_\phi)^N (\zeta)(\xi)|\ls_N \Big(\frac{\beta}{|s_{\lambda,\omega}|}\Big)^N, \qquad N=1,2.
   \end{equation}
   Since the support of the integration above has volume $\approx 2^{-2k}$, \eqref{Best} follows from
   \eqref{eq:ip} and \eqref{eq:diffop}. Hence all that is left is an argument for \eqref{eq:diffop}. 
   
   Let $N=1$. Let
   $(\omega,\omega_2,\omega_3)$ be an orthonormal basis of $\R^3$. For
   $\xi$ in the support of the integration we have
   \[
   \frac{\xi}{|\xi|} = \omega + \mathcal{O}(2^{-k}) \omega_2 +
   \mathcal{O}(2^{-k}) \omega_3 + \mathcal{O}(2^{-2k}), \qquad
   \frac{|\xi|}{\la \xi \ra_k}= \lambda + \mathcal{O}(2^{-2k}),
   \]
   where we recall that $\lambda=\lambda(k)=\frac1{\sqrt{1+2^{-2k}}}$.
   Using these facts we obtain
   \begin{equation*}
     \begin{split}
       \partial_\omega \phi & = \omega \cdot (z + r \frac{\xi}{\la \xi
         \ra_k})
       = \omega \cdot z + r \frac{|\xi|}{\la \xi  \ra_k} + \mathcal{O}(2^{-2k}) \\
       & = \omega \cdot z + r \lambda + \mathcal{O}(2^{-2k}) =
       \frac{s_{\lambda, \omega}}{\beta\sqrt{1+\lambda^2}} +
       \mathcal{O}(2^{-2k})
     \end{split}
   \end{equation*}
   Therefore we obtain $|\partial_\omega\phi | \gs \frac{|s_{\lambda,
       \omega}|}{\beta} \gg 2^{-2k}$. In particular it follows that
       \begin{equation} \label{part1}
       |\frac{\partial_\omega \zeta}{\partial_\omega \phi}| \ls \frac{\beta}{|s_{\lambda,\omega}|}.
       \end{equation}
      where we used that $|\partial_\omega \zeta| \ls 1$. In addition, we have
   \[
   \partial_\omega^2 \phi(\xi) = \partial_\omega \Big(r \frac{\omega
     \cdot \xi}{\la \xi \ra_k}\Big) = r \left( \frac{\omega \cdot
       \omega}{\la \xi \ra_k} - \frac{(\omega \cdot \xi)^2}{\la \xi
       \ra_k^3}\right) = \frac{r}{\la \xi \ra_k} \left(1- (
     \frac{\omega \cdot \xi}{\la \xi \ra_k})^2 \right)
   \]
   from which, using the above arguments, we conclude that in the
   domain of integration we have $| \partial_\omega^2 \phi | \ls
   2^{-2k}$.  This allows us to estimate
   \[
   |\partial_\omega \Big(\frac1{\partial_\omega\phi}\Big)| \ls
   \frac{2^{-2k}}{|\partial_\omega \phi|^2} \ls
   \frac{1}{|\partial_\omega\phi|} \ls \frac{\beta}{|s_{\lambda,
       \omega}|}.
   \]
   From this and \eqref{part1} we obtain \eqref{Best} for $N=1$.
   Now let $N=2$ and compute
   \begin{align*}
    (d^*_\phi)^2 \zeta =\partial_\omega\Big(\frac{1}{\partial_\omega \phi
     }\partial_\omega\frac{\zeta}{\partial_\omega \phi}\Big)=
     \frac{\partial_\omega^2 \zeta}{(\partial_\omega \phi)^2}-3
     \frac{\partial_\omega \zeta \partial_\omega^2
       \phi}{(\partial_\omega \phi)^3}
     -\frac{\zeta \partial_\omega^3\phi}{(\partial_\omega \phi)^3}
     +3\frac{\zeta (\partial_\omega^2\phi)^2}{(\partial_\omega
       \phi)^4}
   \end{align*}
   We compute
   \[
   \partial_\omega^3 \phi =\frac{3r}{\la \xi\ra_k^5}\Big((\omega \cdot
   \xi)^3-(\omega \cdot \xi) \la \xi\ra_k^2\Big)=\mathcal{O}(2^{-2k}).
   \]
   Recalling that $|\partial_\omega\phi | \gs \frac{|s_{\lambda,
       \omega}|}{\beta} \gg 2^{-2k}$, $|\partial_\omega^2 \phi|\ls 2^{-2k}$ and
   $|\partial_\omega^N \zeta|\ls_N 1$ we conclude that
   \[
     |(d^*_\phi)^N |
     \ls \frac{\beta^2}{|s_{\lambda,\omega}|^2}+\frac{2^{-2k}\beta^3}{|s_{\lambda,\omega}|^3}
     +\frac{2^{-4k}\beta^4}{|s_{\lambda,\omega}|^4}
     \ls  \frac{\beta^2}{|s_{\lambda,\omega}|^2}.
   \]
This finishes the proof of \eqref{eq:diffop} and, in turn, the proof of \eqref{eq:ang2}.
 \end{proof}

 We end this section with the proof of \eqref{MStr}. Since there are
 $\approx 2^{2(k-l)}$ caps $\ka \in \mathcal{K}_{k}$ such that $P_\ka
 f \ne 0$, we obtain from \eqref{eq:estr}
 \begin{align*}
   &\sum_{\ka \in \mathcal{K}_{k}} \| e^{it \la D \ra} \tilde
   P_{\kappa}f \|_{L^2_{t_{k,\ka}} L^\infty_{x_{k,\ka}}} \ls 2^{k-l}
   \left( \sum_{\ka \in \mathcal{K}_{k}} \|e^{it \la D \ra}  \tilde P_{\kappa}f \|_{L^2_{t_{k,\ka}} L^\infty_{x_{k,\ka}}}^2 \right)^\frac12\\
   \ls{} & 2^{k-l} \left( \sum_{\ka \in \mathcal{K}_{k}} \|\tilde
     P_{\kappa}f\|_{L^2_x}^2 \right)^\frac12 \ls{} 2^{k-l} \| f
   \|_{L^2_x}.
 \end{align*}

\subsection{Energy estimates in the $(\lambda,\omega)$
   frames}\label{Energy}
 Given a pair $(\lambda,\omega)$ with $\lambda \in \R$ and $\om \in
 \mathbb{S}^2$ we recall that we defined
 \[\Theta_{\lambda,\om}=\frac{1}{\sqrt{1+\lambda^2}}(\lambda,\om),\quad
 \Theta_{\lambda,\om}^\perp=\frac{1}{\sqrt{1+\lambda^2}}(-1,\lambda
 \omega)\] to be two orthogonal vectors in $\R^4$. This can be
 completed to an orthonormal basis in $\R^4$ by considering any two
 vectors $\Theta_{2,\om}=(0,\omega_2)$ and
 $\Theta_{3,\om}=(0,\omega_3)$ such that $(\omega, \omega_2,
 \omega_3)$ form a positively oriented orthonormal basis in $\R^3$.

 With respect to this basis, understanding the
 vectors $\Theta_{\lambda,\om}$,
 $\Theta_{\lambda,\om}^\perp$, $\Theta_{2,\om}$, $\Theta_{3,\om}$ as column
 vectors, we introduce the new coordinates
 $t_{\lambda,\om},x_{\lambda,\om}$, with
 $x_{\lambda,\om}=(x^1_{\lambda,\om},x^2_{\om},x^3_{\om})$, defined by
 \[
 \begin{pmatrix}
   t_{\lambda,\om} \\ x^1_{\lambda,\om} \\ x^2_{\om} \\ x^3_{\om} \\
 \end{pmatrix}
 =\begin{pmatrix} \Theta_{\lambda,\om} & \Theta_{\lambda,\om}^\perp &
   \Theta_{2,\om} & \Theta_{3,\om}
 \end{pmatrix}^t
 \begin{pmatrix}
   t \\ x_1 \\ x_{2} \\ x_3 \\
 \end{pmatrix}
 \]
 In many of the computations we will write
 $x'_{\om}=(x^2_{\om},x^3_{\om})$.

 We denote by $(\tau_{\lambda,\om}, \xi_{\lambda,\om})$ the
 corresponding Fourier variables which are given by
 \[
 \begin{pmatrix}
   \tau_{\lambda,\om} \\ \xi^1_{\lambda,\om} \\ \xi^2_{\om} \\ \xi^3_{\om} \\
 \end{pmatrix}
 =\begin{pmatrix} \Theta_{\lambda,\om} & \Theta_{\lambda,\om}^\perp &
   \Theta_{2,\om} & \Theta_{3,\om}
 \end{pmatrix}
 \begin{pmatrix}
   \tau \\ \xi_1 \\ \xi_{2} \\ \xi_3 \\
 \end{pmatrix}
 \]
 where we also write $\xi'_{\om}=(\xi^2_{\om},\xi^3_{\om})$. In the
 following theorem and its proof we set $B_{k,\ka}=B^+_{k,\ka}$ and $\tilde{B}_{k,\ka}=\tilde{B}^+_{k,\ka}$.
 \begin{thm} \label{thm:Energy} Let $k, j \geq 100$, $0 \leq l \leq
   \min(j,k)-10$ and $\ka \in \mathcal{K}_l$.  Let
   $\Theta_{\lambda,\om}$ be a direction with
   $\lambda=\lambda(j)=\frac1{\sqrt{1+2^{-2j}}}$, and we assume
   $\alpha=\dist(\om, \ka)$ satisfies $2^{-3-l} \leq \alpha \leq
   2^{3-l}$.

{\rm i)}  If $f \in L^2(\R^3)$ has the property that $\hat f$ is
     supported in $A_{k,\ka}$, then for the free solution the
     following holds true
     \begin{equation} \label{DH} \alpha \| e^{ it \la D \ra}
       f\|_{L^\infty_{t_{\lambda,\om}}L^2_{x_{\lambda,\om}}} \ls \| f
       \|_{L^2}.
     \end{equation}

{\rm ii)}  Let $\hat g$ be supported in the set $B_{k,\ka}$ and $g \in
L^1_{t_{\lambda,\om}}L^2_{x_{\lambda,\om}}$. Then, the solution $u$ of the inhomogeneous
     equation
     \begin{equation} \label{inheq} (i \partial_t + \la D \ra) u = g,
       \quad u(0)=0,
     \end{equation}
satisfies the estimate
     \begin{equation} \label{DIH} \alpha \| u
       \|_{L^\infty_{t_{\lambda,\om}}L^2_{x_{\lambda,\om}}} \ls
       \alpha^{-1} \| g
       \|_{L^1_{t_{\lambda,\om}}L^2_{x_{\lambda,\om}}}.
     \end{equation}

{\rm iii)} Under the hypothesis of Part ii) the
     solution $u$ can be written as
     \begin{equation}
       u(t) =e^{ it \la D \ra} \tilde v_0
       + \int_{-\infty}^\infty u_{s}(t) \chi_{t_{\lambda,\om} > s} ds
     \end{equation}
     where  $u_s(t)= e^{ it \la D \ra} v_s$ (homogeneous solution in the original coordinates) and
     \begin{equation}
       \| \tilde v_0 \|_{L^2_x} + \int_{-\infty}^\infty \| v_s \|_{L^2_x} ds
       \ls \al^{-1} \| g \|_{L^1_{t_{\lambda,\om}}L^2_{x_{\lambda,\om}}}. 
     \end{equation}
     In addition $\hat v_s$ and $\hat{\tilde v}_0$ are supported in 
     $\tilde A_{k,\ka}$.
 \end{thm}

 \begin{proof} i) The space-time Fourier of $w(t,x)=e^{ i t \la D \ra}
   f(x)$ is given by the distribution $\F w= \hat f d \sigma$ where $d
   \sigma(\tau,\xi) = \delta_{\tau= \sqrt{|\xi|^2+1}}$ is comparable
   with the standard measure on the surface $\tau=
   \sqrt{|\xi|^2+1}$. We change the variables: $(\tau,\xi) \rightarrow
   (\tau_{\lambda,\om},\xi_{\lambda,\om})$ where $\xi_{\lambda,\om}
   =(\xi^1_{\lambda,\om},\xi'_{\lambda,\om})$. The goal is to write
   $\hat f d \sigma = F
   \delta_{\tau_{\lambda,\om}=h(\xi_{\lambda,\om})}$. We then would
   have
   \begin{equation} \label{Fest} \| F \|_{L^2_{\xi_{\lambda,\om}}} \ls
     (1+\|\nabla h \|_{L^\infty})^\frac12 \| f \|_{L^2}
   \end{equation}
   where the $L^\infty$ norms is taken on the support of $F$.

   The equation of the characteristic surface $\tau= \sqrt{|\xi|^2+1}$
   can be rewritten as \[\tau^2-|\xi|^2-1=0.\] In the new frame this
   takes the form
   \[
   \frac1{\lambda^2+1}(\lambda \tau_{\lambda,\om} -
   \xi^1_{\lambda,\om})^2 - \frac1{\lambda^2+1} (\tau_{\lambda,\om} +
   \lambda \xi^1_{\lambda,\om})^2 - |\xi'_{\lambda,\om}|^2 -1 =0.
   \]
   We solve this equation for $\tau_{\lambda,\om}$, hence we rewrite
   it as follows
   \begin{equation} \label{quadeq} \frac{\lambda^2-1}{\lambda^2+1}
     (\tau_{\lambda,\om})^2 - \frac{4 \lambda}{\lambda^2+1}
     \tau_{\lambda,\om} \xi^1_{\lambda,\om} +
     \frac{1-\lambda^2}{\lambda^2+1} (\xi^1_{\lambda,\om})^2 -
     |\xi'_{\lambda,\om}|^2 -1 =0.
   \end{equation}
   The solutions of this quadratic equation are given by
   \begin{equation} \label{root}\begin{split} \tau_{\lambda,\om}=
       h^\pm(\xi_{\lambda,\om}) = \frac{2 \lambda \xi^1_{\lambda,\om}
         \pm \sqrt{(\lambda^2+1)^2
           (\xi^1_{\lambda,\om})^2+(\lambda^4-1)(|\xi'_{\lambda,\om}|^2+1)}}{\lambda^2-1}.
     \end{split}
   \end{equation}
   We will identify which one of the two solutions is the correct
   one. The positivity of the discriminant
   $\Delta_{\lambda,\omega}=(\lambda^2+1)^2
   (\xi^1_{\lambda,\om})^2+(\lambda^4-1)(|\xi'_{\lambda,\om}|^2+1)$ is
   implicit, as we know a priori that \eqref{quadeq} has at least one
   solution. We will come back shortly to these issues. We continue
   with the following computation:
   \[
   \begin{split}
     \frac{\partial h^\pm}{\partial \xi^1_{\lambda,\om}}
     & =  \frac{1}{\lambda^2-1}(2\lambda+\frac{(\lambda^2+1)^2\xi^1_{\lambda,\om}}{\pm \sqrt{(\lambda^2+1)^2 (\xi^1_{\lambda,\om})^2+(\lambda^4-1)(|\xi'_{\lambda,\om}|^2+1)}})\\
     & = \frac{1}{\lambda^2-1}(2\lambda+ \frac{(\lambda^2+1)^2\xi^1_{\lambda,\om}}{(\lambda^2-1) \tau_{\lambda,\om} - 2 \lambda \xi^1_{\lambda,\om}}) \\
     & = \frac{2\lambda \tau_{\lambda,\om} + (\lambda^2-1)\xi^1_{\lambda,\om}}{(\lambda^2-1) \tau_{\lambda,\om} - 2 \lambda \xi^1_{\lambda,\om}} \\
     & = - \frac{\xi^1_{\lambda,-\om}}{\tau_{\lambda,-\om}}
   \end{split}
   \]
   In a similar manner we obtain $ \nabla_{\xi'_\om} h^\pm =
   (\lambda^2+1) \frac{\xi'_{\lambda,\om}}{\tau_{\lambda,-\om}}$, from
   which, using \eqref{Fest}, it follows
   \begin{equation} \label{inter} \| e^{ it \la D \ra}
     f\|_{L^\infty_{t_{\lambda,\om}}L^2_{x_{\lambda,\om}}} \ls \left(
       1+ \sup_{\xi \in A_{k,\ka}} \frac{2^k}{|\tau_{\lambda,-\om}|}
     \right)^\frac12 \| f \|_{L^2}.
   \end{equation}

   To finish the argument we need a lower bound for
   $|\tau_{\lambda,-\om}|$. We provide below lower bounds for
   $\Delta_{\lambda,\omega}$ and $\tau_{\lambda,-\om}$ for $(\tau,\xi)
   \in B_{k,\ka}$, as these more general bounds are needed in Part
   ii).

   For $(\tau,\xi) \in B_{k,\ka}$ it holds that $\tau -
   \sqrt{|\xi|^2+1} = \epsilon(\tau,\xi)$ with $|\epsilon(\tau,\xi)|
   \leq 2^{k-2l-10}$, hence
   \begin{align*}
   \tau_{\lambda,-\om} &= \lambda \tau - \xi \cdot \omega = \lambda
   \sqrt{|\xi|^2+1} + \lambda \epsilon - \xi \cdot \omega\\
   &= |\xi| \Big(\lambda \sqrt{1+|\xi|^{-2}} + \frac{\lambda \epsilon}{|\xi|} - \frac{\xi \cdot \omega}{|\xi|}\Big)
 \end{align*}
 Given the hypothesis of the Theorem, we obtain $1- 2^{-2l-6}
   \leq \frac{\xi \cdot \omega}{|\xi|} \leq 1- 2^{-2l+6}$, $\frac{|\lambda \epsilon|}{|\xi|} \leq 2^{-2l-8}$, and
   $|\lambda \sqrt{1+|\xi|^{-2}} - 1| \leq 2^{-2\min(j,k)+2}$. Thus
   we conclude that $\tau_{\lambda,-\om} \approx 2^k \alpha^2$ and
   $\tau_{\lambda,-\om} \geq 2^{k-2} \alpha^2$.

   In particular, using \eqref{inter} we obtain \eqref{DH}. Since the
   solutions in \eqref{root} can be recast in the form
   $\tau_{\lambda,-\om}=\pm \sqrt{\Delta_{\lambda,\om}}$ and we just
   proved that $ \tau_{\lambda,-\om} > 0$ in $B_{k,\ka}$, it follows
   that the solutions $h^+$ in \eqref{root} correspond to the choice
   of the surface $\tau = \sqrt{|\xi|^2+1}$.

   We now continue with the more general bounds for
   $\Delta_{\lambda,\om}$ in the set $B_{k,\ka}$.  Since $|\tau - \la
   \xi \ra| \leq 2^{k-10} \alpha^2$ hence $ |\tau^2 - |\xi|^2 - 1| \ls
   2^{2k-8} \alpha^2$ or equivalently, $\tau^2 - |\xi|^2 -
   1=\epsilon(\tau,\xi)$ with $|\epsilon(\tau,\xi)| \ls 2^{2k-8}
   \alpha^2$.  We rewrite the equation in characteristic coordinates
   as above, to obtain
   \[
   \tau_{\lambda,-\om}^2= \Delta_{\lambda,\om} + (1-\lambda^4)\epsilon
   \]
   We have already shown that $\tau_{\lambda,-\om} \geq 2^{k-2}
   \alpha^2$ and since $|(1-\lambda^4)\epsilon| \leq 2^{2k-6} \alpha^2
   |1-\lambda| \leq 2^{2k-6} \alpha^4$, it follows that
   $\Delta_{\lambda,\om} \geq 2^{2k-4} \alpha^4$ in $B_{k,\ka}$. A
   similar argument proves $\Delta_{\lambda,\om} \approx 2^{2k}
   \alpha^4$ in $B_{k,\ka}$.

   ii) On the Fourier side the inhomogeneous problem \eqref{inheq}
   becomes
   \[
   (-\tau + \la \xi \ra) \hat u = \hat g
   \]
   which we rewrite as follows
   \[
   (\tau^2-|\xi|^2-1) \hat u = (-\tau -\la \xi \ra)\hat g :=\hat G.
   \]
   Due to the localization in $B_{k,\ka}$ it follows that $\hat G =
   a\hat g$ where
   \[
   a(\tau,\xi)= (-\tau - \la \xi \ra)
   \tilde{\chi}_k(\xi)\tilde{\eta}_\ka \tilde{\chi}_{\leq k-2l}(\tau -
   \la \xi \ra)
   \]
   has the property $\| \F^{-1}_{t,x} a \|_{L^1_{t,x}} \ls 2^k$. From
   this it follows that
   \begin{equation}\label{Gest}
     \| G \|_{L^1_{t_{\lambda,\om}}L^2_{x_{\lambda,\om}}}  \ls 2^k \| g \|_{L^1_{t_{\lambda,\om}}L^2_{x_{\lambda,\om}}} 
   \end{equation}
 
   In the new coordinates the equation above becomes
   \[
   \frac{\lambda^2-1}{\lambda^2+1}
   (\tau_{\lambda,\om}-h^+(\xi_{\lambda,\om}))(\tau_{\lambda,\om}-h^{-}(\xi_{\lambda,\om}))\hat u
   = \hat G
   \]
   where $h^{\pm}(\xi_{\lambda,\om})$ are the two roots in
   \eqref{root} of the quadratic equation \eqref{quadeq}.  We have
   \begin{align*}
     &|(\lambda^2-1)(\tau_{\lambda,\om}-h^\pm(\xi_{\lambda,\om}))|=| (\lambda^2-1) \tau_{\lambda,\om}
     - 2\lambda \xi^1_{\lambda,\om} \pm \sqrt{\Delta_{\lambda,\om}} | \\
     ={}& |(\lambda^2+1) \tau_{\lambda,-\om} \pm
     \sqrt{\Delta_{\lambda,\om}} |
   \end{align*}
   From part i) we have that $|(\lambda^2+1)\tau_{\lambda,\om} +
   \sqrt{\Delta_{\lambda,\om}}| \approx 2^k \alpha^2$ in $B_{k,\ka}$.
   We then rewrite the equation above as follows
   \[
   (\tau_{\lambda,\om}- h^- (\xi_{\lambda,\om})) \hat u = m^{-1} \tilde
   \chi_{B_{k,\ka}}\hat G
   \]
   where
   $m(\tau_{\lambda,\om},\xi_{\lambda,\om})=\frac{1-\lambda^2}{1+\lambda^2}
   (\tau_{\lambda,\om}-h^+(\xi_{\lambda,\om}))$ and $\tilde
   \chi_{B_{k,\ka}}$ is a smooth function which equals $1$ in
   $B_{k,\ka}$ and is supported in the double of the set $B_{k,\ka}$.
   Taking the inverse Fourier transform with respect to
   $\tau_{\lambda,\om}$ only gives
   \[
   (-i\partial_{t_{\lambda,\om}}-h^-(\xi_{\lambda,\om}))
   \F_{x_{\lambda,\om}} u= K \ast_{t_{\lambda,\om}}
   \F_{x_{\lambda,\om}}G
   \]
   where
   $K(t_{\lambda,\om},\xi_{\lambda,\om})=\mathcal{F}^{-1}_{\tau_{\lambda,\om}}
   (m^{-1}\tilde
   \chi_{B_{k,\ka}})$. A solution for the
   above problem is given by the Duhamel formula
   \begin{equation} \label{tu} \F_{x_{\lambda,\om}} v
     (t_{\lambda,\om},\xi_{\lam,\om})=
     \int_{-\infty}^{t_{\lambda,\om}}
     e^{i(t_{\lam,\om}-s)h^-(\xi_{\lambda,\om})} (K
     \ast_{t_{\lambda,\om}} G)(s,\xi_{\lam,\om}) ds
   \end{equation}
 
   In integral form the kernel $K$ is given by
   \[
   K(t_{\lambda,\om},\xi_{\lambda,\om}) =
   \frac{1+\lambda^2}{1-\lambda^2} \int \frac{e^{it_{\lambda,\om}
       \tau_{\lambda,\om}}}{\tau_{\lambda,\om}-h^+(\xi_{\lambda,\om})}
   \tilde\chi_{B_{k,\ka}}(\tau_{\lambda,\om},
   \xi_{\lambda,\om}) d\tau_{\lambda,\om}
   \]
   We fix $\xi_{\lambda,\om}$ and by using stationary phase it follows
   that
   \[
   |K_\alpha(t_{\lambda,\om},\xi_{\lambda,\om})| \ls_N
   \frac{1}{1-\lambda^2} \la t_{\lambda,\om} (1-\lambda^2)^{-1} 2^k
   \alpha^2 \ra^{-N}
   \]
   which has the advantage that it holds uniformly with respect to
   $\xi_{\lambda,\omega}$. From this we obtain
   \[
   \| K \|_{L^1_{t_{\lambda,\om}}L^\infty_{\xi_{\lambda,\om}}} \ls
   (2^k \alpha^2)^{-1}.
   \]
   This implies that
   \[
   \| K \ast_{t_{\lambda,\om}} G
   \|_{L^1_{t_{\lambda,\om}}L^2_{x_{\lambda,\om}}} \ls (2^k
   \alpha^2)^{-1} \| G
   \|_{L^1_{t_{\lambda,\om}}L^2_{x_{\lambda,\om}}}.
   \]
   from which, when combined with \eqref{Gest}, we obtain
   \[
   \| v \|_{L^\infty_{t_{\lambda,\om}}L^2_{x_{\lambda,\om}}} \ls
   \alpha^{-2} \| g \|_{L^1_{t_{\lambda,\om}}L^2_{x_{\lambda,\om}}}.
   \]
   Thus we have produced a solution $v$ of the inhomogeneous equation
   \[
   (i \partial_t + \la D \ra)v=g
   \]
   satisfying the bounds in \eqref{DIH} but without satisfying the
   initial condition $v(0)=0$. Therefore we have that
   \[
   u(t) = v(t) - e^{ i t \la D \ra } v(0).
   \]
   We rewrite \eqref{tu} as follows
   \[
   v = \int_{-\infty}^{\infty} v_s \chi_{t_{\lambda,\om} \geq s} ds
   \]
   where $\F_{\xi_{\lambda,\om}} v_s =
   e^{i(t_{\lam,\om}-s)h^-(\xi_{\lambda,\om})} (K
   \ast_{t_{\lambda,\om}} G)(s,\xi_{\lam,\om})$. Thus $v$ is a
   superposition of free waves truncated across the hyperplanes
   $t_{\lambda,\om} = s$. In addition, by reversing the computations
   in part i) we obtain
   \[
   \| v_s \|_{L^\infty_t L^2_x} \ls \alpha^{-1} \| (K
   \ast_{t_{\lambda,\om}} G)(s) \|_{L^2_{x_{\lambda,\om}}}
   \]
   from which it follows
   \[
   \int_{-\infty}^\infty \| v_s \|_{L^\infty_t L^2_x} ds \ls
   \alpha^{-1} \| g \|_{L^1_{t_{\lambda,\om}}L^2_{x_{\lambda,\om}}}.
   \]
   In particular this implies that
   \[
   \| v(0) \|_{L^2_x} \ls \alpha^{-1} \| g
   \|_{L^1_{t_{\lambda,\om}}L^2_{x_{\lambda,\om}}}
   \]
   and by invoking part i) we obtain
   \[
   \| e^{ i t \la D \ra} v(0)
   \|_{L^\infty_{t_{\lambda,\om}}L^2_{x_{\lambda,\om}}} \ls
   \alpha^{-2} \| g \|_{L^1_{t_{\lambda,\om}}L^2_{x_{\lambda,\om}}}
   \]
   which finishes the argument for part ii). In fact this also proves
   part iii) of the Theorem.
 \end{proof}
 
 \subsection{Estimates for the Klein-Gordon equation}\label{subsect:kg}
 Let us specifically describe how the above estimates read in the
 context of the Klein-Gordon equation
\begin{equation}\label{eq:kg-id}
(\Box+m^2)u=g, u(0)=f_0, u_t(0)=f_1,
\end{equation}
where $m \ne 0$ is fixed. The analogue of Theorem \ref{thm:estr} is
 \begin{cor}\label{cor:estr-kg} Let $m \ne 0$.
   Suppose that $u$ is the solution of \eqref{eq:kg-id} with $g=0$ and the initial data $f_0,f_1\in L^2(\R^3)$ satisfy
   \[\supp(\widehat{f_0}),\; \supp(\widehat{f_1})\subset \tilde{A}_k,
   \quad k
   \in \Z.\]

{\rm i)} For all $k \ls 1$,
     \begin{equation} \label{eq:lf-kg} \| u \|_{L^2_t L^\infty_x} \ls
       2^{\frac{k}2} \| f_0 \|_{L^2}+\| f_1 \|_{L^2}
     \end{equation}

{\rm ii)}  For all $k \gs 1$, $\ka \in
     \mathcal{K}_{k}$,
     \begin{equation} \label{eq:estr-kg} 2^{-k} \|P_{\kappa}u
       \|_{L^2_t L^\infty_x} +
       \|P_{\kappa}u\|_{L^2_{t_{k,\ka}}L^\infty_{x_{k,\ka}}} \ls
       \|P_{\kappa} f_0 \|_{L^2}+2^{-k} \|P_{\kappa} f_1 \|_{L^2}
     \end{equation}

{\rm iii)}  For all $k \gs 1$, $1 \leq l \leq k$, $
     \ka_1 \in \mathcal{K}_l$,
     \begin{equation} \label{MStr-kg} \sum_{\ka \in \mathcal{K}_k}
       \|P_{\ka}P_{\kappa_1}u\|_{L^2_{t_{k,\ka}}L^\infty_{x_{k,\ka}}}
       \ls 2^{k-l} \|P_{\kappa_1} f_0\|_{L^2}+2^{-l} \|P_{\kappa_1}
       f_1\|_{L^2}.
     \end{equation}
 \end{cor}
 The proof is obvious, see \eqref{eq:kg-prop}.  Of course, there is
 also an analogue of Theorem \ref{thm:Energy} for \eqref{eq:kg-id}.
 \begin{cor} \label{cor:EnergyKG} Let $k, j \geq 100$, $0 \leq l \leq
   \min(j,k)-10$ and $\ka \in \mathcal{K}_l$.  Let
   $\Theta_{\lambda,\om}$ be a direction with
   $\lambda=\lambda(j)=\frac1{\sqrt{1+2^{-2j}}}$, and we assume
   $\alpha=\dist(\om, \ka)$ satisfies $2^{-3-l} \leq \alpha \leq
   2^{3-l}$.

{\rm i)} If $f_0,f_1 \in L^2(\R^3)$ have the property that $\hat f_0, \hat f_1$ are
     supported in $A_{k,\ka}$, then the solution $u$ to \eqref{eq:kg-id} with $g=0$ satisfies 
     \begin{equation} 
      \alpha \| u \|_{L^\infty_{t_{\lambda,\om}}L^2_{x_{\lambda,\om}}} 
     \ls \| f_0  \|_{L^2} + 2^{-k} \| f_1  \|_{L^2}.
     \end{equation}

{\rm ii)} Assume that $f_0=f_1=0$
     and let $\hat g$ be supported in the set $B^+_{k,\ka} \cup
     B^-_{k,-\ka}$ and $g \in L^1_{t_{\lambda,\om}}L^2_{x_{\lambda,\om}}$. Then, the solution $u$ of \eqref{eq:kg-id} satisfies
     \begin{equation}  \alpha \| u
       \|_{L^\infty_{t_{\lambda,\om}}L^2_{x_{\lambda,\om}}} \ls
      2^{-k} \alpha^{-1} \| g
       \|_{L^1_{t_{\lambda,\om}}L^2_{x_{\lambda,\om}}}
     \end{equation}

{\rm iii)} Under the hypothesis of Part ii) the
     solution $u$ can be written as
     \begin{equation}
       u(t) = v(t)
       + \int_{-\infty}^\infty u_{s}(t) \chi_{t_{\lambda,\om} > s} ds
     \end{equation}
     where $v$ and $u_s$ are homogeneous solutions of the Klein-Gordon
     equation (in the original coordinates) and
     \begin{equation}
     \begin{split}
         & \int_{-\infty}^\infty (\| u_s(0) \|_{L^2_x} + 2^{-k}\| \partial_t u_s(0) \|_{L^2_x}) ds \\
         & + \| v(0) \|_{L^2_x} +  2^{-k}\| \partial_t v(0) \|_{L^2_x}
       \ls 2^{-k} \al^{-1} \| g \|_{L^1_{t_{\lambda,\om}}L^2_{x_{\lambda,\om}}}. 
     \end{split}
     \end{equation}
     In addition, $\hat u_s$ and $\hat v$ are supported in $\tilde{B}^+_{k,\ka} \cup \tilde{B}^-_{k,-\ka}$.
 \end{cor}

\section{Setup of the cubic Dirac}\label{sect:setupD}
As written in \eqref{eq:dirac} the cubic Dirac equation has a linear
part whose coefficients are matrices. We rewrite \eqref{eq:dirac} as a
new system whose linear parts are the two half Klein-Gordon equations,
see \eqref{CDsys} below.

In the new setup it is possible to identify a null-structure in the
nonlinearity, which is very similar to the ideas for the
Dirac-Klein-Gordon system presented in \cite[Section 2 and
3]{dfs}. This will play a key role in overcoming some logarithmic
divergences in the bilinear estimates.  The main difference is that we
keep the mass term inside the operator.

\subsection{Reduction}\label{subsect: red}
The cubic Dirac equation can be written as
\begin{equation} \label{CDmod} -i (\partial_t + \alpha \cdot \nabla +
  i \beta ) \psi= \la \psi, \beta \psi \ra \beta \psi.
\end{equation}
where $\beta=\gamma^0$ and $\alpha^j=\gamma^0 \gamma^j$ and $\alpha
\cdot \nabla = \alpha^j \partial_j$.  The new matrices satisfy
\begin{equation} \label{al} \alpha^j \alpha^k + \alpha^k \alpha^j = 2
  \delta^{jk} I_4, \qquad \alpha^j \beta + \beta \alpha^j=0.
\end{equation}
There is one more computation which we will use in this section,
namely
\begin{equation} \label{alal} \alpha^j \alpha^k = \delta^{jk} + i
  \epsilon^{jkl} S^l
\end{equation}
where $\epsilon^{jkl}=1$ if $(j,k,l)$ is an even permutation of
$(1,2,3)$, $\epsilon^{jkl}=-1$ if $(j,k,l)$ is an odd permutation of
$(1,2,3)$ and $\epsilon^{jkl}=0$ otherwise (when it contains repeated
indexes).  The matrices $S^l$ are defined by
\[
S^l = \left( \begin{array}{cc} \sigma^l & 0 \\ 0 &
    \sigma^l \end{array} \right).
\]

Following \cite[Section 2]{dfs} we decompose the spinor field relative
to a basis of the operator $\alpha \cdot \nabla + i \beta $ whose
symbol is $\alpha \cdot \xi + \beta$. Since $(\alpha \cdot \xi +
\beta)^2= (|\xi|^2+1)I$, the eigenvalues are $\pm \la \xi \ra$.  We
introduce the projections $\Pi_{\pm}(D)$ with symbol
\[
\Pi_\pm(\xi)=\frac12 [I \mp \frac{1}{\la \xi \ra} ( \xi \cdot \alpha +
\beta)].
\]
In comparison to \cite[formula (2.2)]{dfs}, note that in the definition of $\Pi_\pm$ we chose the opposite sign for internal consistency purposes. The key identity is
\[
-i (\alpha \cdot \nabla + i \beta ) = \langle D \rangle
(\Pi_-(D)-\Pi_+(D))
\]
where $\langle D \rangle$ has symbol $\sqrt{|\xi|^2+1}$. The following
identity, which can be verified easily at the level of the symbols,
will be important in our computations:
\[
\Pi_\pm(D) \beta = \beta (\Pi_\mp(D)\mp\frac{\beta}{\la D \ra}).\]

We then define $\psi_\pm=\Pi_\pm(D) \psi$ and split $\psi=\psi_+ +
\psi_-$. By applying the operators $\Pi_\pm(D)$ to the cubic Dirac
equation we obtain the following system of equations
\begin{equation} \label{CDsys}
  \begin{cases}
    (i\partial_t + \langle D \rangle) \psi_+ = -\Pi_+(D) (\la \psi, \beta \psi \ra \beta \psi) \\
    (i\partial_t - \langle D \rangle) \psi_- = -\Pi_-(D) (\la \psi,
    \beta \psi \ra \beta \psi).
  \end{cases}
\end{equation}
This system will replace \eqref{eq:dirac} as the object of our
research for the rest of the paper. It is obvious from the form of the
operators $\Pi_\pm$ that $\| \psi \|_{X} \approx \| \psi_+ \|_{X}+\|
\psi_- \|_{X}$ for many reasonable function spaces $X$.  In particular
we use it for $X=H^1(\R^3)$ so that we conclude that the initial data
for \eqref{CDsys} satisfies $\psi_\pm(0) \in H^1(\R^3)$.

\subsection{Null Structure}\label{subsect:nullst}
There is a subtle null structure hidden in the system \eqref{CDsys},
which we describe next.  This is again inspired by the work on the
Dirac-Klein-Gordon system in \cite{dfs}.

We start with $\la \psi, \beta \psi \ra$ which, in our decomposition,
is rewritten as
\[
\begin{split}
  \la \psi, \beta \psi\ra & = \la (\Pi_+(D) \psi_+ + \Pi_-(D) \psi_-, \beta (\Pi_+(D) \psi_+ + \Pi_-(D)) \psi_- \ra \\
  & = \la \Pi_+(D) \psi_+, \beta \Pi_+(D) \psi_+ \ra + \la \Pi_-(D) \psi_-, \beta \Pi_-(D)) \psi_- \ra \\
  & + \la \Pi_+(D) \psi_+, \beta \Pi_-(D) \psi_- \ra+ \la \Pi_-(D)
  \psi_-, \beta \Pi_+(D) \psi_+ \ra
\end{split}
\]
The following Lemma analyses the symbols of the bilinear operators
above, which is very similar to \cite[Lemma 2]{dfs} and its proof.
\begin{lem} The following holds true
  \begin{equation} \label{PiPi}
    \begin{split}
      \Pi_\pm(\xi) \Pi_\mp(\eta) & = \mathcal{O}(\angle (\xi,\eta)) + \mathcal{O}(\la \xi \ra^{-1} + \la \eta \ra^{-1}) \\
      \Pi_\pm(\xi) \Pi_\pm(\eta) & = \mathcal{O}(\angle (-\xi,\eta)) +
      \mathcal{O}(\la \xi \ra^{-1} + \la \eta \ra^{-1})
    \end{split}
  \end{equation}
\end{lem}

\begin{proof} We use the notation $\hat \xi := \frac{\xi}{|\xi|}$.
  Since $\frac{\xi}{\la \xi \ra}= \frac{\xi}{|\xi|} + \mathcal{O}(\la
  \xi \ra^{-1})$, and similarly for $\eta$, it follows, cp.\
  \cite[p.886]{dfs}, that
  \begin{align*}
    4 \Pi_\pm(\xi) \Pi_\mp(\eta) 
    =  & [I \mp \frac{1}{\la \xi \ra}(\xi \cdot \alpha + \beta)][I \pm \frac{1}{\la \eta \ra}(\eta \cdot \alpha + \beta)] \\
    ={}& I - \hat \xi_j \hat \eta_k \alpha^j \alpha^k \mp (\hat \xi - \hat \eta) \cdot \alpha + \mathcal{O}(\la \xi \ra^{-1} + \la \eta \ra^{-1}) \\
    ={}& (1-\hat \xi \cdot \hat \eta) I - i (\hat\xi \times\hat \eta) \cdot S \mp (\hat \xi - \hat \eta) \cdot \alpha + \mathcal{O}(\la \xi \ra^{-1} + \la \eta\ra^{-1}) \\
    ={}& \mathcal{O}(\angle (\xi,\eta)) + \mathcal{O}(\la \xi \ra^{-1}
    + \la \eta \ra^{-1})
  \end{align*}
  where in passing from the second to the third line we have used
  \eqref{al} and \eqref{alal}.  The second estimate in \eqref{PiPi}
  follows from the first and the fact that
  $\Pi_\pm(\xi)=\Pi_\mp(-\xi)+\mathcal{O}(\la \xi\ra^{-1})$.
\end{proof}

We now explain why the above result plays the role of a null
structure. Taking the spatial Fourier transform yields
\[\mathcal{F}_x\la \Pi_+(D) \psi_1, \beta
\Pi_+(D) \psi_2 \ra(\nu)=\int\limits_{\nu=\xi+\eta} \la \Pi_+(\xi)
\widehat{\psi_1}(\xi), \beta \Pi_+(\eta) \widehat{\psi_2}(\eta) \ra\]
where we suppose that $\widehat{\psi_1},\widehat{\psi_2}$ are
supported at high frequencies $|\xi|, |\eta| \gg 1$. In this regime
the equation is of wave type and it is well-known that the strongest
interactions are the parallel ones, i.e.\ when $\angle
(\xi,\eta)=0$. On the other hand we have
\begin{align*}
  &\la \Pi_+(\xi) \widehat{\psi_1}(\xi), \beta \Pi_+(\eta)\widehat{\psi_2}(\eta) \ra\\
  ={} &\la \Pi_+(\xi) \widehat{\psi_1}(\xi), \Pi_-(\eta) \beta
  \widehat{\psi_2}(\eta) \ra
  -\la \Pi_+(\xi) \widehat{\psi_1}(\xi), \frac{1}{\la \eta \ra}  \widehat{\psi_2}(\eta) \ra\\
  ={}& \la \Pi_-(\eta) \Pi_+(\xi) \widehat{\psi_1}(\xi), \beta
  \widehat{\psi_2}(\eta) \ra -\la \Pi_+(\xi) \widehat{\psi_1}(\xi),
  \frac{1}{\la \eta \ra} \widehat{\psi_2}(\eta) \ra
\end{align*}
From the above computation it follows that, when $\angle (\xi,\eta)=0$, 
\[\Pi_-(\eta) \Pi_+(\xi) =
\mathcal{O}(\la \xi \ra^{-1} + \la \eta \ra^{-1}),
\] 
thus greatly
improving the structure of the bilinear form.

\section{Function Spaces}\label{fspaces}
Based on the structures developed in Section \ref{sect:le} we are now
ready to define the function spaces in which we will perform the
Picard iteration for \eqref{CDsys}. Notice that there are similarities
to the function spaces used in the wave map problem
\cite{Kr03,tao,tat}, which we highlight by using a similar
notation. 

For $1\leq p\leq \infty$, $b \in \R$, we define
\[\|f\|_{\dot{X}^{\pm,b,p}}=\big\|\big(2^{bm}\|Q_{m}^\pm
f\|_{L^2}\big)_{m \in \Z }\big\|_{\ell^p_m},\]

For the
low frequency part we define
\[
\| f \|_{S^\pm_{\leq 99}} = \| f \|_{L^\infty_t L^2_x} + \| f
\|_{L^2_t L^\infty_x} + \|f \|_{X^\pm,\frac12,\infty}
+ \sup_{m \in \Z} 2^{m} \|Q_{m}^\pm f \|_{L^\frac43_t L^2_x}.
\]

For the large frequencies, that is $k \geq 100$, the norm has a
multiscale structure.  For $l \leq k-10$ and $\ka \in \mathcal{K}_l$
we define
\[
  \| f \|_{S^\pm[k,\ka]} = \| f \|_{L^\infty_t L^2_x} +\sup_{j \geq l + 10} \sup_{\ka_1 \in \mathcal{K}_l: \atop
    2^{-l-3}\leq \dist(\kappa, \kappa_1) \leq 2^{-l+3} } 2^{-l} \| f
  \|_{L^\infty_{t^\pm_{j,\ka_1}}L^2_{x^\pm_{j,\ka_1}}}
\]
 and
\begin{equation}\label{eq:sk}
  \begin{split}
    \| f \|_{S^\pm_k} =& \|f \|_{L^\infty_t L^2_x}
    +\|f\|_{\dot{X}^{\pm,\frac12,\infty}} + 2^{-\frac{k}4}\sup_{m \in \Z} 2^{m}
    \|Q_{m}^\pm f \|_{L^\frac43_t L^2_x}\\
{}&+
    \Big(\sum_{\ka \in \mathcal{K}_k}  2^{-2k}\| P_{\ka} f
  \|^2_{L^2_t L^\infty_x} +\| P_{\ka} f \|^2_{L^2_{t^\pm_{k,\ka}} L^\infty_{x^\pm_{k, \ka}}}\Big)^{\frac12} \\
    {}&+\sup_{1\leq l\leq k-10} \Big(\sum_{\ka \in \mathcal{K}_l}\|
    Q_{\prec k-2l}^\pm P_{\ka} f \|^2_{S^\pm[k; \ka]}\Big)^{\frac12}
  \end{split}
\end{equation}

The resolution space corresponding to regularity at the level of
$H^\sigma(\R^3)$ is the closed subspace of $C(\R,H^\sigma(\R^3))$ defined by the norm
\[
\| f \|_{S^{\pm, \sigma}} = \| P_{\leq 99} f \|_{S_{\leq 99}^\pm} +
\Big(\sum_{k \geq 100} 2^{2k \sigma} \| P_k f
\|^2_{S^\pm_k}\Big)^{\frac12}.
\]

Now we turn our attention to the construction of the space for the nonlinearity.
For the low frequency part we define
\[
\| f \|_{N^{\pm,at}_{\leq 99}} = \inf_{f=f_1+f_2} \Big\{\|f_1 \|_{\dot{X}^{\pm,-\frac12,1}} + \| f_2 \|_{L^1_tL^2_x} \Big\} .
\]
and
\[
\| f \|_{N^\pm_{\leq 99}} = \| f \|_{N^{\pm,at}_{\leq 99}} 
+ \| f \|_{L^\frac43_t L^2_x}.
\]
An important property of these spaces is
\begin{equation} \label{duall} 
S_{\leq 99}^\mp \subset (N_{\leq 99}^{\pm,at})^*
  \subset S_{\leq 99}^{\mp,w}.
\end{equation}
where $(N_{\leq 99}^{\pm,at})^*$ is the dual of $N_{\leq 99}^{\pm,at}$ 
and $S^{\pm,w}_{\leq 99}$ is endowed with the norm
\begin{equation}\label{eq:skweakl}
    \| f \|_{S^{\pm,w}_{\leq 99}}= \| f \|_{L^\infty_t L^2_x} + \| f \|_{X^{\pm,\frac12,\infty}}. 
\end{equation} 

Next let $k \geq 100$. For $l \leq k-10$ we consider $\ka \in
\mathcal{K}_l$ and define
\[
\| f \|_{N^\pm[k,\ka]} = \inf \Big\{ 2^{l} \sum_{(j,\kappa_1)} \|
f_{j,\ka_1} \|_{L^1_{t^\pm_{j,\ka_1}}L^2_{x^\pm_{j,\ka_1}}}\colon
f=\sum_{(j,\kappa_1)} f_{j,\ka_1} \Big\}
\]
where the infimum is taken over pairs $(j,\kappa_1)$ with $l \leq j
-10$ and $\ka_1 \in \mathcal{K}_l$ with $2^{-3} \leq 2^l
\dist(\ka_1,\ka) \leq 2^{3}$. Then we define the space for the
following atomic structure
\begin{equation}\label{eq:n-atom}
  \begin{split}
    \| f \|_{N_k^{\pm,at}}=& \inf_{f=f_1+f_2+\sum_{1 \leq l \leq k-10} g_{l} } \Big\{\|f_1 \|_{\dot{X}^{\pm,-\frac12,1}} + \| f_2 \|_{L^1_tL^2_x} \\
    {}&+ \sum_{1 \leq l \leq k-10} \Big( \sum_{\ka \in \mathcal{K}_l}
    \| P_{\ka} g_{l} \|_{N^\pm[k, \ka]}^2 \Big)^\frac12\Big\}
  \end{split}
\end{equation}
where the atoms $g_l$ in the above decomposition are assumed to be
localized at frequency $2^k$ and modulation $\ll 2^{k-2l}$, more
precisely that $\tilde{Q}_{\prec k-2l}^\pm \tilde{P}_k g_l = g_l$.

One important remark should be made about the third component in
$N_k^{\pm,at}$, i.e. the $\sum_{1 \leq l \leq k-10} g_{l}$, which we will henceforth call the cap-localized structure.  The atoms
$g_l$ are localized in frequency and modulation, while when they are
measured in $N^\pm[k,\ka]$ the components in the decomposition there
$g_l = \sum_{(j,\ka_1)} g_{l,j,\ka_1}$ are not assumed to keep that
localization. However, by applying the operator $\tilde{Q}_{\prec
  k-2l}^\pm \tilde{P}_{k,\ka}$ to the decomposition and using part i) in 
Lemma \ref{stable} below one obtains a new decomposition with similar
norm. From now on we assume that the decomposition above
comes with the correct frequency and modulation localization.

An important property of this construction is that
\begin{equation} \label{dual} S_k^\mp \subset (N_{k}^{\pm,at})^*
  \subset S_k^{\mp,w}
\end{equation}
where $(N_{k}^{\pm,at})^*$ is the dual of $N_{k}^{\pm,at}$ and
$S_k^{\pm,w}$ is endowed with the norm
\begin{equation}\label{eq:skweak}
    \| f \|_{S^{\pm,w}_k}= \| f \|_{L^\infty_t L^2_x} + \| f \|_{X^{\pm,\frac12,\infty}} 
    + \sup_{1 \leq l \leq k} \Big( \sum_{\ka \in \mathcal{K}_l} \|
    Q_{\prec k-2l}^\pm P_{\ka} f \|^2_{S^\pm[k;\ka]}\Big)^{\frac12}
\end{equation} 
and the embeddings are continuous, i.e.
\[ \|f\|_{S_k^{\mp,w}}\ls \|f\|_{(N^{\pm,at}_{k})^{*}}\ls
\|f\|_{S_k^\mp}.\]

For high frequencies, the space for dyadic pieces of the nonlinearity is
the following
\[
\| f \|_{N^\pm_k} = \| f \|_{N^{\pm,at}_k} + 2^{-\frac{k}4} \| f
\|_{L^\frac43_t L^2_x}.
\]
The space for the nonlinearity at regularity $H^\sigma$ is the
following
\[
\| f \|_{N^{\pm, \sigma}} = \| P_{\leq 99} f \|_{N_{\leq 99}^\pm} +
\Big(\sum_{k \geq 100} 2^{2k \sigma} \| P_k f
\|^2_{N^\pm_k}\Big)^{\frac12}.
\]

We now turn our attention to the relevance of the above structures for
the equations we study.  Our first result is of technical nature and
it says that certain frequency and modulation localization operators
preserve the structures involved above.

\begin{lem}\label{stable}
{\rm i)} For all $k \geq 100$, $1 \leq l \leq k$, 
$\kappa \in \mathcal{K}_l$, the operators $\tilde P_{k,\ka}$ and $\tilde Q_{\prec k-2l}^\pm \tilde
  P_{k,\kappa}$ have bounded kernel in $L^1_x$, respectively $L^1_{t,x}$. 
  As a consequence, they are uniformly bounded on all $L^pL^q$ in all frame choices.

{\rm ii)} For all $k,j \geq 100$, $1 \leq l \leq
  \min(j,k) -10$, $\kappa, \kappa_1 \in \mathcal{K}_l$ such that
  $2^{-3-l} \leq \dist(\kappa,\kappa_1) \leq 2^{3-l}$, the operators
  $ \tilde Q_m^\pm \tilde P_{k,\kappa}\text{ for }m
  \leq k-2l$ are bounded on the spaces
  $L^1_{t^\pm_{j,\ka_1}}L^2_{x^\pm_{j,\ka_1}}$. 
  
{\rm iii)} For all $k \geq 100, 1\leq l \leq k, \ka \in \mathcal{K}_l$, and functions $u$ localized at frequency $2^k$, we have
  \begin{equation} \label{sta} \| \left( \Pi_\pm(D) - \Pi_\pm(2^{k}
      \omega(\ka)) \right) P_{\ka} u \|_{S} \ls 2^{-l} \|
    P_{\ka} u \|_{S}
  \end{equation}
   for $S \in \{ S_k^\pm, S_k^{\pm,w} \}$.
\end{lem}
\begin{proof} i) The kernel of the operator $\tilde P_{k,\ka}$ is given by 
$\mathcal{F}_x^{-1} (\tilde \eta_\ka \tilde \chi_k)$ and it is a straightforward 
exercise to prove that it belongs to $L^1_x$. Since 
\[
\tilde P_{k,\ka} u = \mathcal{F}_x^{-1} (\tilde \eta_\ka \tilde \chi_k) \ast_x u
\]
the boundedness of $\tilde P_{k,\ka}$ on all $L^pL^q$ spaces follows from the boundedness
of its kernel in $L^1_x$. 

Next, we prove the statement for the operator $\tilde Q_{\prec k-2l}^+ \tilde
  P_{k,\kappa}$. With $a_{l,k,\ka}(\tau,\xi)= \tilde \chi_{\leq k-2l}(\tau - \la \xi
    \ra) \tilde \eta_\ka \tilde \chi_k$ and $R=\mathcal{F}^{-1} (a_{l,k,\ka})$ we have
  \[
  Q^+_{\prec k-2l} P_{\ka} u = R \ast Q^+_{\prec k-2l} P_k u.
  \]
 Since $a$ is a smooth approximation of the characteristic function of a rectangular parallelepiped 
 (of sizes $2^{k} \times 2^{k-2l} \times 2^{k-l} \times 2^{k-l}$ in the direction of 
 $(\tau_{k,\ka},\xi^1_{k,\ka},\xi^2_{k,\ka},\xi^3_{k,\ka})$), it is a straightforward exercice  
 to prove that $\| R \|_{L^{1}_{t,x}} \ls 1$. The boundedness statement follows from the above.

ii)  We give the proof for the operator $\tilde Q_{m}^+ \tilde P_{k,\kappa}$, which is a Fourier
  multiplier whose symbol $a_{m,k,\ka}(\tau,\xi)=\tilde \chi_{m} (\tau - \la \xi \ra) \tilde \chi_k (\xi) \tilde
  \eta_\ka(\xi)$ satisfies
  \[
  |\partial_{\tau_{j,\ka_1}}^\beta a_{m,k,\ka}| \ls
  (2^{m+2l})^{-\beta}.
  \]
  The inverse Fourier transform of $a_{m,k,\ka}$ with respect to
  $\tau_{j,\ka_1}$ satisfies
  \[
  |K_{l,k,\ka}(t_{j,\ka_1},\xi_{j,\ka_1})| \ls_N
  2^{m+2l}(1+|t_{j,\ka_1}|2^{m+2l})^{-N}, \text{ for any }N \in \N.
  \]
  From this we obtain the uniform bound
  \[
  \| K_{l,k,\ka} \|_{L^1_{t_{j,\ka_1}}L^\infty_{\xi_{j,\ka_1}}} \ls 1.
  \]
  On the other hand we have
  \[
  \mathcal{F}_{\xi_{j,\ka_1}} ( \tilde Q^+_{m} \tilde P_{k,\ka}
  f) = K_{l,k,\ka} *_{t_{j,\ka_1}}
  \mathcal{F}_{\xi_{j,\ka_1}} f,
  \]
  where one performs convolution with respect to $t_{j,\ka_1}$
  variable only. From the last two statements, the conclusion follows.
 
iii)  We prove the statement for the $+$ choice above and $S=S_k^+$, 
the proof for the other choices being similar. A similar argument to the one used in
i) shows that the operators $\left( \Pi_+(D) - \Pi_+(2^{k} \omega(\ka)) \right) P_{k,\ka}$
and $\left( \Pi_+(D) - \Pi_+(2^{k} \omega(\ka)) \right) Q^+_{\prec k-2l} P_{k,\ka}$ are, up to picking a factor of $2^{-l}$, uniformly bounded on each component.
\end{proof}

The main result of this section is the following Proposition.
\begin{pro} \label{linl} For all $g\in N_k^\pm$ and initial data $u_0\in L^2(\R^3)$, both localized at (spatial) frequency $2^k$, $k \geq
  100$, the solution $u$ of
  \begin{equation} \label{ng} (i \partial_t \pm \la D \ra) u = g,
    \quad u(0)=u_0,
  \end{equation}
  belongs to $S^\pm_k$ and the following estimate holds true:
  \begin{equation} \label{mlin} \| u \|_{S_k^\pm} \ls \| g
    \|_{N_k^\pm} + \| u_0 \|_{L^2}.
  \end{equation}
\end{pro}

\begin{proof}
  To simplify the exposition we write the argument for the $+$ choice
  above.  The argument is organized as follows. In Part 1 we consider $g \in
  N^{+,at}_k$ and we derive all the properties in $S_k$ for $u$, except
  the $L^\frac43_t L^2_x$ structure. Since the $N^{+,at}_k$ contains
  three type of atoms, we split the argument in three cases. In Part 2,
  we prove that if $g \in 2^{\frac{k}{4}} L^\frac43_t L^2_x$, then we
  obtain the similar structure for $Q_m^+ u$.

Further, since all estimates in $S_k^+$ were provided for homogeneous solutions in Section \ref{sect:le},
it suffices to provide the argument for $u_0=0$. We note that the homogenous solutions belong to
 the kernel of the operators $Q_m^+$, hence the $\dot X^{+,\frac12,\infty}$ and $L^\frac43_t L^2_x$
 components are vacuous for them.  

\underline{Part 1) $g \in N^{+,at}_k$.}
{\it Case a) $g \in L^1_t L^2_x$.} The solution is given by
  \[
  u(t) = - e^{ i t\la D \ra} \int_{-\infty}^0 e^{-
    i s \la D \ra} g(s) ds + \int e^{i(t -s)\la D \ra} g(s) \chi_{t >
    s} ds.
  \]
  Hence $u$ is a superposition of homogeneous solutions with $L^2$
  data which are truncated across hyperplanes $t > s$. The $L^\infty_tL^2_x$ bound is obvious. Theorem \ref{thm:estr} and Theorem \ref{thm:Energy} i) imply the end-point Strichartz and energy estimates. The estimate in
  $\dot{X}^{+,\frac12,\infty}$ is proved as follows. Inserting the
  modulation operator $Q_m^+$ into the equation we obtain
  \[
  (i \partial_t + \la D \ra) Q_m^+ u = Q_m^+ g.
  \]
Let $D_t=i\partial_t$. Then,
\[
Q_m^+=e^{it\la D\ra}\chi_m(D_t)e^{-it\la D\ra}
\]
which yields
  \begin{equation}\label{eq:q-sol}
  D_t \chi_m(D_t) e^{-it\la D\ra}u = \chi_m(D_t)e^{-it\la D\ra}g.
 \end{equation}
Now, the kernel of $D_t^{-1}\chi_m(D_t)$ satisfies
\begin{equation}\label{eq:q-ker}
 \Big\|\F_t^{-1}\big(\frac{\chi_m}{\tau}\big)\Big\|_{L^q(\R)}\ls 2^{-\frac{m}{q}}
 \end{equation}
for all $1\leq q \leq \infty$, hence
  \begin{align*}
 & \|Q_m^+ u\|_{L^2}=\|D_t^{-1}\chi_m(D_t)e^{-it\la D\ra}g\|_{L^2} \\
\ls{}&
\Big\|\F_t^{-1}(\frac{\chi_m}{\tau})\Big\|_{L^2_t}\|e^{-it\la D\ra}g\|_{L^1_t L^2_x}
\ls{} 2^{-\frac{m}{2}} \|g \|_{L^1_t L^2_x}.
\end{align*}

{\it Case b) $g \in \dot{X}^{+,-\frac12,1}$.}  Let $v$ defined by 
$\hat v=\frac{\hat g}{\tau-\la \xi \ra}$.
As defined now, $v$ may not even be a distribution. 
Using $g \in 2^{\frac{k}{4}}L^{4/3}_tL^2_x$ and the frequency localization of $g$, it follows
from a Sobolev embedding that $ g \in L^2$. Thus $g=\sum_{m \in \Z}
Q_m^+ g$, and it follows further that 
$\hat v=\sum_{m \in \Z} \frac{\chi_m(\tau-\la \xi \ra) \hat g}{\tau-\la \xi \ra}$.

Then, by \eqref{eq:q-ker} with $q=1$,
\begin{align*}
 & \sum_{m\in \Z} 2^{\frac{m}{2}}\|Q_m^+ v\|_{L^2}=\sum_{m\in Z} 2^{\frac{m}{2}}\|D_t^{-1}\chi_m(D_t)e^{-it\la D\ra}g\|_{L^2} \\
\ls{}&\sum_{m\in \Z}2^{\frac{m}{2}}
\Big\|\F_t^{-1}(\frac{\chi_m}{\tau})\Big\|_{L^1_t}\|(\chi_{m-1}(D_t)+\chi_{m}(D_t)+\chi_{m+1}(D_t))e^{-it\la D\ra}g\|_{L^2}\\
\ls{}& \sum_{m\in \Z} 2^{-\frac{m}{2}} \|Q_m^+ g \|_{L^2},
\end{align*}
hence $v \in \dot{X}^{+,\frac12,1}$ and $\|v\|_{L^\infty_t L^2_x}\ls \|v\|_{\dot{X}^{+,\frac12,1}}\ls
\|g\|_{\dot{X}^{+,-\frac12,1}} $; in particular we upgraded $v$ to a tempered distribution. Further, $v$ can be written as
\[
v=\sum_{m \in \Z} \int e^{it\tau} e^{it\la D\ra}\tilde{v}_m(\tau)d\tau, \, \text{ where } \tilde{v}_m=\F_t(e^{-it\la D\ra}g)\frac{\chi_m}{\tau},
\]
i.e.\ as a superposition of modulated homogeneous solutions.
Due to the estimate
\[
\sum_{m \in \Z} \int \|\tilde{v}_m(\tau)\|_{L^2_x} d\tau\ls \sum_{m \in \Z} 2^{-\frac{m}{2}}\|Q_m^+ g\|_{L^2}=\|g\|_{\dot{X}^{+,-\frac12,1}}
\] the end-point Strichartz and energy estimates for $v$ follow from Theorem \ref{thm:estr} and Theorem \ref{thm:Energy} i). The only problem is that while $v$ satisfies the inhomogeneous equation \eqref{ng},
it does not have to satisfy the initial condition. On the other hand 
\[
u= v - e^{it\la D \ra} v(0)
\]
becomes a solution to \eqref{ng} (with $u_0=0$) and since $\| v(0) \|_{L^2_x} \ls \| g \|_{\dot X^{+,-\frac12,1}}$, 
\eqref{mlin} follows in this case.

{\it Case c) $g$ belongs to the cap-localized structure.}
  Given the $l^1$ structure in the $l$ parameter, it suffices to
  establish the estimates for fixed $l$. For each $\ka \in
  \mathcal{K}_l$ we have the decomposition
  \begin{equation} \label{decnon}
  P_{\ka} g_l = \sum_{(j,\kappa_1)} g_{j,\ka_1}
  \end{equation}
  where we recall that we can choose $g_{j,\ka_1}$ such that $\tilde
  Q_{\prec k-2l}^+ \tilde P_{k,\ka} g_{j,\ka_1} = g_{j,\ka_1}$. Using
  part iii) of Theorem \ref{thm:Energy} with $g_{j,\ka_1}$ as forcing,
  we obtain that the solution generated satisfies
  \[
  \| u_{j,\ka_1} \|_{S^+[k,\ka]} \ls \| g_{j,\ka_1}
  \|_{L^1_{t_{j,\ka_1}}L^2_{x_{j,\ka_1}}}
  \]
  and has Fourier support in the set $\tilde B_{k,\ka}$. If $u_{\ka}$
  is the solution of the equation with forcing $P_\ka g_l$, then by
  adding all the components in the decomposition of $g_l$ gives the
  following estimate
  \[
  \| u_\ka \|_{S^+[k,\ka]} \ls \sum_{(j,\ka_1)} \| g_{j,\ka_1}
  \|_{L^1_{t_{j,\ka_1}}L^2_{x_{j,\ka_1}}}
  \]
  and that $u_\ka$ has Fourier support in the set $\tilde
  B_{k,\ka}$. In the last step we need to perform the summation with
  respect to $\ka \in \mathcal{K}_l$. Given that each $u_\ka$ is
  supported in $\tilde B_{k,\ka}$, the $L^\infty_t L^2_x$ and the
  end-point Strichartz estimate follow. Concerning the cap-localized
  structure, it is easy to see that one obtains the $S^+[k,\ka']$
  structures with $\ka' \in \mathcal{K}_{l'}$ with $l' \geq l$. For
  the case when $l' \leq l$, one splits
  \[
  P_{\ka'} u = \sum_{\ka \in \mathcal{K}_l} \tilde P_{\ka} P_{\ka'} u
  \]
  and uses the almost orthogonality of $P_{\ka'} u_\ka$, $\ka \in
  \mathcal{K}_l$ with respect to $\xi_{j,\ka_1}$ to obtain
  \[
  \| P_{\ka'} u \|^2_{L^\infty_{t_{j,\ka_1}} L^2_{x_{j,\ka_1}}} \ls
  \sum_{\ka \in \mathcal{K}_l} \| P_{\ka'} u_\ka
  \|^2_{L^\infty_{t_{j,\ka_1}} L^2_{x_{j,\ka_1}}}.
  \]
  We now prove that $u \in \dot{X}^{+,\frac12,\infty}$. We start from the decomposition
  \eqref{decnon}. From this we obtain
  \[
  \begin{split}
    \| Q_m^+ g_{j,\ka_1} \|_{L^2_{t,x}} = \| \mathcal{F} (Q_m^+
    g_{j,\ka_1}) \|_{L^2_{\tau_{j,\ka_1},\xi_{j,\ka_1}}}
    & \ls 2^{\frac{m+2l}2} \| \mathcal{F} (Q_m^+ g_{j,\ka_1}) \|_{L^2_{\xi_{j,\ka_1}}L^\infty_{\tau_{j,\ka_1}}} \\
    & \ls 2^{\frac{m+2l}2} \| Q_{m}^+ g_{j,\ka_1} \|_{L^1_{t_{j,\ka_1}}L^2_{x_{j,\ka_1}}} \\
    & \ls 2^{\frac{m+2l}2} \| g_{j,\ka_1}
    \|_{L^1_{t_{j,\ka_1}}L^2_{x_{j,\ka_1}}}.
  \end{split}
  \]
  In the above we have used that the size of the support of Fourier
  transform of $Q_{m}^+ g_{j,\ka_1}$ in the direction of
  $\tau_{j,\ka_1}$ is $\approx 2^{m+2l}$ and part ii) of Lemma \ref{stable}. We sum the above estimates
  with respect to $(j,\ka_1)$ to obtain
  \[
  \| Q_{m}^+ P_{k,\ka} g_l \|_{L^2} \ls 2^{\frac{m}2} \| Q_{\prec
    k-2l}^+ P_{k,\ka} g_l \|_{N[k,\ka]}.
  \]
  Finally, we sum the above with respect to $\ka \in \mathcal{K}_l$ to
  conclude with
  \[
  2^{-\frac{m}2} \| Q_{m}^+ g_l \|_{L^2} \ls \left( \sum_{\ka \in
      \mathcal{K}_l} \| Q_{\prec k-2l}^+ P_{k,\ka} g_l \|^2_{L^2}
  \right)^\frac12.
  \]
  Since this is uniform with respect to $m \leq 2k-l$ we obtain that $g \in \dot X^{+,-\frac12,\infty}$. 
  Since $\mathcal{F}( Q_m^+ u) = \frac{1}{\tau-\la \xi \ra} \mathcal{F}( Q_m^+ g)$, the estimate for $u$
  in  $\dot{X}^{+,\frac12,\infty}$ follows.

  \underline{Part 2) $g$ belongs to $2^{\frac{k}{4}} L^\frac43_t L^2_x$.}
From \eqref{eq:q-sol} it follows that
  \begin{align*}
 & \|Q_m^+ u\|_{L^{4/3}_t L^2_x}=\|D_t^{-1}\chi_m(D_t)e^{-it\la D\ra}g\|_{L^{4/3}_t L^2_x} \\
\ls{}&
\Big\|\F_t^{-1}(\frac{\chi_m}{\tau})\Big\|_{L^1_t}\|e^{-it\la D\ra}g\|_{L^{4/3}_t L^2_x}
\ls{} 2^{-m} \|g \|_{L^{4/3}_t L^2_x},
\end{align*}
where we used \eqref{eq:q-ker} with $q=1$, and this finishes our proof.
\end{proof}

\begin{cor} \label{corlin} For all $u_0 \in H^\sigma(\R^3)$ and $g \in N^{\pm,\sigma}$, there exists a unique solution $u\in S^{\pm,\sigma}$ of
  \eqref{ng}, and the following estimate holds true
  \begin{equation} \label{mlins} \| u \|_{S^{\pm,\sigma}} \ls \| g
    \|_{N^{\pm,\sigma}} + \| u_0 \|_{H^\sigma}.
  \end{equation}
\end{cor}

\begin{proof} The claim follows from its dyadic versions for high frequencies ($k \geq 100$), which is precisely Proposition \ref{linl}. The low frequency part is standard, except the $L^\frac43_t L^2_x$ part which is established 
as in Part 2) above.  Alternatively it is an easy exercise to work out the whole argument following the same steps as for the high frequency case.
\end{proof}

\section{Bilinear estimates}\label{sect:bil-est}
In this section we derive the main bilinear $L^2_{t,x}$-type estimate
for functions in our spaces.  As a convention, throughout the rest of
the paper $u$'s will denote complex scalars, $u: \R \times \R^3
\rightarrow \C$, while $\psi$'s will denote complex vectors $\psi : \R
\times \R^3 \rightarrow \C^4$. To make the exposition simpler we will abuse
notation and set $S_{99}^\pm:=S^\pm_{\leq 99}$. 

The main result of this section is the following

\begin{pro}\label{Pbil}
{\rm i)} For all $k_1,k_2 \geq 99$ and $\psi_1 \in
    S^\pm_{k_1}$, $\psi_2 \in
    S^{\pm,w}_{k_2}$, where $\psi_j$ localized at frequency $2^{k_j}$ for $j=1,2$, the
    following holds true:
    \begin{equation} \label{bil2} \big\| \la \Pi_\pm(D) \psi_1 , \beta
      \Pi_\pm(D) \psi_2 \ra \big\|_{L^2} \ls 2^{k_1} \| \psi_1
      \|_{S_{k_1}^\pm} \| \psi_2 \|_{S_{k_2}^{\pm,w}},
    \end{equation}

{\rm ii)} If in addition $l \leq \min(k_1,k_2)$, then
    \begin{equation} \label{bil3}
      \begin{split}
        & \Bigg\| \sum_{\ka_1,\ka_2 \in \mathcal{K}_{l}: \atop
          \dist(\pm\ka_1,\pm\ka_2)
          \ls 2^{-l}} \la \Pi_\pm(D) \tilde P_{\ka_1} \psi_1, \beta \Pi_\pm(D) \tilde P_{\ka_2}  \psi_2 \ra \Bigg\|_{L^2} \\
        \ls{}& 2^{k_1-l} \| \psi_1 \|_{S_{k_1}^\pm} \| \psi_2
        \|_{S_{k_2}^{\pm,w}}.
      \end{split}
    \end{equation}
    In both of the above estimates the sign of each $\Pi_\pm$ and $\pm
    \ka_j$ is
    chosen to be consistent with the one of the corresponding
    $S^\pm_{k_j}$.

 {\rm iii)} Let $2 < q \leq \infty$. For all $100 \leq k_1
    \leq k_2$ and $u_1 \in S^+_{k_1}$, $u_2 \in S^{+,w}_{k_2}$, each localized at frequency
    $2^{k_1}$ resp.\ $2^{k_2}$,  the following holds true:
    \begin{equation} \label{bil4} \| u_1 \cdot \bar u_2 \|_{L^2_t
        L^q_x} \ls_q 2^{k_1} 2^{3(\frac12 - \frac1q) k_2} \| u_1
      \|_{S_{k_1}^+} \| u_2 \|_{S_{k_2}^{+,w}}.
    \end{equation}
    The same result holds true for $u_1 \in S^-_{k_1}$, $u_2 \in S^{-,w}_{k_2}$.
\end{pro}
 
As an immediate consequence \eqref{bil4} we note the following
Strichartz type estimate.
\begin{cor}\label{cor:str-off}
  Let $4 < q \leq \infty$. For all $k \geq 100$,
  \begin{equation} \label{Str4} \| \tilde P_{k} u \|_{L^4_t L^q_x}
    \ls_q 2^{\frac{k}2} 2^{\frac32(\frac12 - \frac2q) k} \| \tilde P_k u
    \|_{S^\pm_k}.
  \end{equation}
\end{cor}
By interpolation one can easily obtain all the "off the line''
Strichartz estimates $L^p_t L^q_x$ with $p \geq 4$, following closely
the ideas of \cite{Kr03,Kr04,tao} in the context of wave maps. In the
case of wave maps, it has been observed later in \cite[Section
5.4]{StTa10} that the usual "on the line'' Strichartz estimates such
as $L^4_{t,x}$ hold true in these spaces as well, but this is a little
more difficult to prove and we do not need it here.
 
The low frequency counterpart of \eqref{cor:str-off} is, for all $4 < q \leq \infty$,
\begin{equation} \label{Strlow}
\| \tilde P_{k} u \|_{L^4_t L^q_x} \ls  \| \tilde P_{\leq 99} u \|_{L^4_{t,x}} \ls \| \tilde P_{\leq 99} u \|_{S^\pm_{\leq 99}}.
\end{equation}
which is easily obtained from the $L^4_{t,x}$ using Sobolev embedding. The latter is obtained
using interpolation between the $L^2_t L^\infty_x$ and $L^\infty_t L^2_x$ components of $S^\pm_{\leq 99}$.  
 
\begin{proof}[Proof of Proposition \ref{Pbil}] 
  To make the exposition easier, we choose to prove all the estimates
  for the $+$ choice in all terms.  A careful examination of the
  argument reveals that the other choices follow in a similar manner.
  The focus of the argument is on the high frequency interactions, that is
  $\min(k_1,k_2) \geq 100$. It will be obvious that when  $\min(k_1,k_2) = 99$,
  the argument carries on and in fact it becomes simpler. Note that \eqref{bil3}
  does not say anything new in the case $\min(k_1,k_2) = 99$, while \eqref{bil4}
  is not even stated in this case. 
  
  We will reduce \eqref{bil2},\eqref{bil3} and \eqref{bil4} to the
  following claim: For all $u_1,u_2$ be localized at frequencies
  $2^{k_1}$, respectively $2^{k_2}$, and $|l_1-l_2| \leq 2$ with $l_1
  \leq \min(k_1,k_2)$ the following estimate holds true:
  \begin{equation} \label{bas1} \sum_{\ka_1 \in
      \mathcal{K}_{l_1},\ka_2 \in \mathcal{K}_{l_2}} \| \tilde
    P_{\ka_1}u_1 \tilde P_{\ka_2} u_2 \|_{L^2} \ls 2^{k_1} \| u_1 \|_{S^+_{k_1}} \| u_2
    \|_{S_{k_2}^{+,w}}.
  \end{equation}
  where the above sum is restricted to the range $\dist(\ka_1, \ka_2)
  \approx 2^{-l_1}$ or $\dist( \ka_1, \ka_2) \ls 2^{-l_1}$ in the case
  $|l_1-\min(k_1,k_2)| \leq 2$.

  \underline{First case: $k_1 \leq k_2$.} If $l_1 \leq k_1 - 10$, then
  \begin{equation*}
    \sum_{\ka_1 \in
      \mathcal{K}_{l_1},\ka_2 \in \mathcal{K}_{l_2}}\| \tilde P_{\ka_1} u_1 \cdot \tilde P_{\ka_2} u_2  \|_{L^2} 
    \leq A_0+A_1+A_2+A_3.
  \end{equation*}
We will provide estimates for each contribution.
\begin{align*}
    A_0:=&
\sum_{\ka_1 \in
      \mathcal{K}_{l_1},\ka_2 \in \mathcal{K}_{l_2}}
\| Q_{\succeq k_1-2l_1} \tilde  P_{\ka_1} u_1 \|_{L^4_tL^\infty_x} \| \tilde P_{\ka_2}
    Q_{\succeq k_1-2l_2} u_2 \|_{L^4_tL^2_x} \\
    \ls{} & 2^{\frac{3k_1-2l_1}2} \Big(\sum_{\ka_1 \in
      \mathcal{K}_{l_1}}
\| Q_{\succeq k_1-2l_1} \tilde  P_{\ka_1} u_1 \|^2_{L^4_tL^2_x}\Big)^{\frac12} \Big(\sum_{\ka_2 \in
      \mathcal{K}_{l_2}}\| \tilde P_{\ka_2}
    Q_{\succeq k_1-2l_2} u_2 \|_{L^4_tL^2_x}^2\Big)^{\frac12}.
  \end{align*}
Now, we use
\begin{align*}
& \Big(\sum_{\ka_1 \in
      \mathcal{K}_{l_1}}
\| Q_{\succeq k_1-2l_1} \tilde  P_{\ka_1} u_1
\|^2_{L^4_tL^2_x}\Big)^{\frac12}\ls \sum_{m \succeq k_1-2l_1} \Big(\sum_{\ka_1 \in
      \mathcal{K}_{l_1}}
\| Q_{m} \tilde  P_{\ka_1} u_1
\|^2_{L^4_tL^2_x}\Big)^{\frac12}\\
\ls & \sum_{m \succeq k_1-2l_1} 2^{\frac{m}{4}}\Big(\sum_{\ka_1 \in
      \mathcal{K}_{l_1}}
\| Q_{m} \tilde  P_{\ka_1} u_1
\|^2_{L^2_tL^2_x}\Big)^{\frac12}\\
\ls & \sum_{m \succeq k_1-2l_1} 2^{\frac{m}{4}}
\| Q_{m} u_1
\|_{L^2_tL^2_x}\ls  2^{-\frac{k_1-2l_1}{4}}\|Q_{\succeq k_1-2l_1} u_1\|_{\dot X^{+,\frac12,\infty}}
\end{align*}
to complete the argument as follows:
\begin{align*}
A_0\ls & 2^{\frac{3k_1-2l_1}2}  2^{-\frac{k_1-2l_1}{2}}\|Q_{\succeq k_1-2l_1} u_1\|_{\dot X^{+,\frac12,\infty}}\|Q_{\succeq k_1-2l_2} u_2\|_{\dot X^{+,\frac12,\infty}}\\
    \ls{}& 2^{k_1} \| \tilde P_{\ka_1} u_1 \|_{S^+_{k_1}} \| \tilde
    P_{\ka_2} u_2 \|_{S_{k_2}^{+,w}},
  \end{align*}
and 
  \begin{align*}
    A_1:=& \sum_{\ka_1 \in
      \mathcal{K}_{l_1},\ka_2 \in \mathcal{K}_{l_2}}\| Q_{\prec
      k_1-2l_1} \tilde P_{\ka_1} u_1 \|_{L^\infty} \| \tilde P_{\ka_2}
    Q_{\succeq k_1-2l_2} u_2 \|_{L^2} \\
\ls & \Big(\sum_{\ka_1 \in
      \mathcal{K}_{l_1}}\| Q_{\prec k_1-2l_1} \tilde P_{\ka_1} u_1 \|^2_{L^\infty} \Big)^{\frac12}\Big(\sum_{\ka_2 \in
      \mathcal{K}_{l_2}}\| \tilde P_{\ka_2}
    Q_{\succeq k_1-2l_2} u_2 \|^2_{L^2} \Big)^{\frac12}\\
\ls & 2^{\frac{3k_1-2l_1}2} \Big(\sum_{\ka_1 \in
      \mathcal{K}_{l_1}}\| Q_{\prec k_1-2l_1} \tilde P_{\ka_1} u_1
    \|^2_{L^\infty_tL^2_x} \Big)^{\frac12}
2^{-\frac{k_1-2l_2}2} \|
    Q_{\succeq k_1-2l_2}
    u_2 \|_{\dot{X}^{+,\frac12,\infty}} \\
    \ls{}& 2^{k_1} \| u_1 \|_{S^+_{k_1}} \| u_2 \|_{S_{k_2}^{+,w}},
  \end{align*}
  and
  \begin{align*}
    A_2:= &\sum_{\ka_1 \in
      \mathcal{K}_{l_1},\ka_2 \in \mathcal{K}_{l_2}} \| Q_{\succeq
      k_1-2l_1} \tilde  P_{\ka_1} u_1 \|_{L^2_t
      L^\infty_{x}} \| Q_{\prec k_1-2l_2} \tilde P_{\ka_2} u_2
    \|_{L^\infty_{t}L^2_{x}}\\
 \ls &\Big(\sum_{\ka_1 \in
      \mathcal{K}_{l_1}} \| Q_{\succeq k_1-2l_1}\tilde P_{\ka_1} u_1 \|^2_{L^2_t
      L^\infty_{x}}\Big)^{\frac12} \Big(\sum_{\ka_2 \in
      \mathcal{K}_{l_2}}\| Q_{\prec k_1-2l_2} \tilde P_{\ka_2} u_2
    \|^2_{L^\infty_{t}L^2_{x}}\Big)^{\frac12}\\
    \ls{}& 2^{\frac{3k_1-2l_1}2} \Big(\sum_{\ka_1 \in
      \mathcal{K}_{l_1}}\| Q_{\succeq k_1-2l_1} \tilde P_{\ka_1}
    u_1 \|_{L^2_{t,x}}^2\Big)^{\frac12} \| u_2 \|_{S_{k_2}^{+,w}}
    \\
    \ls{}& 2^{k_1} \| u_1 \|_{S^+_{k_1}} \| u_2 \|_{S_{k_2}^{+,w}},
  \end{align*}
  as well as
  \begin{align*}
    A_3&:= \sum_{\ka_1 \in
      \mathcal{K}_{l_1},\ka_2 \in \mathcal{K}_{l_2}} \sum_{\ka \in
      \mathcal{K}_{k_1}} \|  P_\ka Q_{\prec k_1-2l_1} \tilde P_{\ka_1}
    u_1 \|_{L^2_{t_{k_1,\ka}}L^\infty_{x_{k_1,\ka}}} \| Q_{\prec
      k_1-2l_2} \tilde P_{\ka_2} u_2
    \|_{L^\infty_{t_{k_1,\ka}}L^2_{x_{k_1,\ka}}}\\
 & \ls  2^{l_1} \sum_{\ka_1 \in
      \mathcal{K}_{l_1},\ka_2 \in \mathcal{K}_{l_2}} \| Q_{\prec
      k_1-2l_2} \tilde P_{\ka_2} u_2
    \|_{S[k_2,\ka_2]} \sum_{\ka \in
      \mathcal{K}_{k_1}} \|  P_\ka \tilde P_{\ka_1}
    u_1 \|_{L^2_{t_{k_1,\ka}}L^\infty_{x_{k_1,\ka}}} \\
& \ls  2^{k_1} \sum_{\ka_1 \in
      \mathcal{K}_{l_1},\ka_2 \in
      \mathcal{K}_{l_2}} \| Q_{\prec
      k_1-2l_2} \tilde P_{\ka_2} u_2
    \|_{S[k_2,\ka_2]} \Big(\sum_{\ka \in
      \mathcal{K}_{k_1}} \|  P_\ka \tilde P_{\ka_1}
    u_1 \|_{L^2_{t_{k_1,\ka}}L^\infty_{x_{k_1,\ka}}}^2\Big)^{\frac12} \\
    & \ls  2^{k_1}\Big( \sum_{\ka_2 \in
      \mathcal{K}_{l_2}} \| Q_{\prec
      k_1-2l_2} \tilde P_{\ka_2} u_2
    \|_{S[k_2,\ka_2]}^2\Big)^{\frac12} \Big(\sum_{\ka \in
      \mathcal{K}_{k_1}} \|  P_\ka u_1 \|_{L^2_{t_{k_1,\ka}}L^\infty_{x_{k_1,\ka}}}^2\Big)^{\frac12}\\
    &\ls 2^{k_1} \| u_1 \|_{S^+_{k_1}} \| u_2 \|_{S_{k_2}^{+,w}}.
  \end{align*}
  If $k_1 -10 \leq l_1 \leq k_1$, then the argument is entirely
  similar, but for the $A_3$ contribution we use $L^2_t L^\infty_x$
  and $L^\infty_t L^2_x$.
    
  \underline{Second case: $k_1 \geq k_2$.} The argument above works
  the same way for $l_1 \leq k_2-10$. Consider now the case $
  k_2 - 10 \leq l_1 \leq k_2$. Again, the contributions analogous to  $A_0$, $A_1$
  and $A_2$ can be treated in the same way (now, the modulation
  threshold is $k_2-2l_j$). In the case of $A_3$ (low
  modulation), we face the problem that $\|P_{\ka_1}
  u_1\|_{L^2_tL^\infty_x}$ gives suboptimal bounds, because $\ka_1\in
  \mathcal{K}_{l_1}$ with $l_1\approx k_2$ instead of $k_1$. Therefore, we decompose
  \[
  \tilde P_{\ka_1} u_1 = \sum_{\ka \in \mathcal{K}_{k_1}} P_{\ka}
  \tilde P_{\ka_1} u_1
  \]
  and note that the interactions $P_{\ka} \tilde P_{\ka_1} u_1 \tilde
  P_{\ka_2} u_2$ are almost orthogonal with respect to $\ka \in
  \mathcal{K}_{k_1}$. Indeed this follows from the fact that both
  $P_{\ka} \tilde P_{\ka_1} u_1$ and $ \tilde P_{\ka_2} u_2$ have
  Fourier-support of size $\approx 1$ in the orthogonal directions to
  $\omega(\ka_2)$. Thus
  \begin{align*}
    &\| \tilde P_{\ka_1}  Q_{\prec
      k_2-2l_1} u_1 \cdot \tilde P_{\ka_2}  Q_{\prec
      k_2-2l_2} u_2 \|_{L^2}^2\\
    \ls{}& \sum_{\ka \in \mathcal{K}_{k_1}} \| P_{\ka} \tilde P_{\ka_1}  Q_{\prec
      k_2-2l_1} u_1 \cdot \tilde P_{\ka_2}  Q_{\prec
      k_2-2l_2} u_2 \|_{L^2}^2 \\
    \ls{}& \sum_{\ka \in \mathcal{K}_{k_1}} \| P_{\ka} \tilde P_{\ka_1} Q_{\prec
      k_2-2l_1} u_1\|_{L^2_t L^\infty_x}^2  \cdot \| \tilde P_{\ka_2}  Q_{\prec
      k_2-2l_2} u_2 \|_{L^\infty_t L^2_x}^2.
  \end{align*}
For the contribution $A_3$, we obtain the bound
\begin{align*}
    & \sum_{\ka_1 \in
      \mathcal{K}_{l_1},\ka_2 \in \mathcal{K}_{l_2}} \Big(\sum_{\ka \in \mathcal{K}_{k_1}} \| P_{\ka} \tilde P_{\ka_1}  Q_{\prec
      k_2-2l_1} u_1\|_{L^2_t L^\infty_x}^2  \cdot \| \tilde P_{\ka_2}  Q_{\prec
      k_2-2l_2} u_2 \|_{L^\infty_t L^2_x}^2\Big)^{\frac12}\\
 & \ls  \sum_{\ka_1 \in
      \mathcal{K}_{l_1},\ka_2 \in \mathcal{K}_{l_2}}  \Big(\sum_{\ka \in
      \mathcal{K}_{k_1}}\|  P_\ka \tilde P_{\ka_1}
    Q_{\prec
      k_2-2l_1} u_1 \|_{L^2_t L^\infty_x}^2\Big)^{\frac12} \| Q_{\prec
      k_2-2l_2} \tilde P_{\ka_2} u_2
    \|_{S[k_2,\ka_2]}\\
& \ls  \Big(\sum_{\ka \in
      \mathcal{K}_{k_1}}\|  P_\ka
    Q_{\prec
      k_2-2l_1} u_1 \|_{L^2_t L^\infty_x}^2\Big)^{\frac12} \Big(\sum_{\ka_2 \in \mathcal{K}_{l_2}} \| Q_{\prec
      k_2-2l_2} \tilde P_{\ka_2} u_2
    \|_{S[k_2,\ka_2]}^2\Big)^{\frac12}\\
    &\ls 2^{k_1} \| u_1 \|_{S^+_{k_1}} \| u_2 \|_{S_{k_2}^{+,w}}.
\end{align*}
The proof of the claim
  \eqref{bas1} is now complete.

  As an immediate consequence of the above argument we obtain
  \begin{equation} \label{bas1a} \sum_{\ka_1 \in
      \mathcal{K}_{l_1},\ka_2 \in \mathcal{K}_{l_2}} \| \tilde
    P_{\ka_1}u_1 \overline{\tilde P_{\ka_2} u_2} \|_{L^2} \ls 2^{k_1}
    \| u_1 \|_{S^+_{k_1}} \|u_2
    \|_{S_{k_2}^{+,w}}.
  \end{equation}

  Now, we turn to the proof of \eqref{bil2}.  Using \eqref{bas1a} we
  claim the following
  \begin{equation} \label{basnul}\begin{split} &\sum_{\ka_1 \in
        \mathcal{K}_{l_1},\ka_2 \in \mathcal{K}_{l_2}} \| \la \Pi_+(D)
      P_{\ka_1} \psi_1, \beta \Pi_+(D) P_{\ka_2} \psi_2 \ra \|_{L^2}\\
      \ls{}& 2^{k_1-l_1} \| \Pi_+(D) P_{k_1,\ka_1} \psi_1
      \|_{S^+_{k_1}} \| \Pi_+(D) P_{k_2,\ka_2} \psi_2
      \|_{S_{k_2}^{+,w}},
    \end{split}
  \end{equation}
  where the sum is restricted to the range $\dist(\ka_1, \ka_2)
  \approx 2^{-l_1}$ or $\dist( \ka_1, \ka_2) \ls 2^{-l_1}$ in the case
  $|l_1-\min(k_1,k_2)| \leq 2$.  To prove \eqref{basnul}, we linearize
  the operator $\Pi_+(D)$ as follows
  \[
  \Pi_+(D) = \Pi_+(2^{k_j} \omega(\ka_j)) + \Pi_+(D) - \Pi_+(2^{k_j}
  \omega(\ka_j))
  \]
  where $j=1,2$. Taking into account \eqref{bas1a} and \eqref{PiPi} we
  obtain
  \begin{align*}
    &\| \la \Pi_+(2^{k_1} \omega(\ka_1)) P_{\ka_1} \psi_1 , \beta
    \Pi_+(2^{k_2} \omega(\ka_2)) P_{\ka_2} \psi_2 \ra
    \|_{L^2} \\
    \ls{}& 2^{k_1-l_1} \| P_{\ka_1} \psi_1 \|_{S^+_{k_1}} \| P_{\ka_2}
    \psi_2 \|_{S_{k_2}^{+,w}}
  \end{align*}
  where we have used $|\angle(\omega(\ka_1),\omega(\ka_2)) | \ls
  2^{-l_1}$ and that
  \[
  \mathcal{O}(2^{-k_1} + 2^{-k_2}) \ls 2^{-\min(k_1,k_2)} \ls
  2^{-l_1}.
  \]
  
  The estimate for the remaining terms follows from using
  \eqref{bas1a} and \eqref{sta}.  By organizing the interacting
  factors based on their angle of interaction we have
  \begin{align*}
    &\| \la \Pi_+(D) \psi_1 , \beta \Pi_+(D) \psi_2 \ra \|_{L^2} \\
    \ls{}& \sum_{|l_1-l_2| \leq 2} \sum_{\ka_1 \in
      \mathcal{K}_{l_1},\ka_2 \in \mathcal{K}_{l_2}} \| \la P_{\ka_1}
    \Pi_+(D) \psi_1 , \beta P_{\ka_2} \Pi_+(D) \psi_2 \ra \|_{L^2}
  \end{align*}
  where the first sum is restricted over the range $1 \leq l_1,l_2
  \leq \min(k_1,k_2)$, and the second sum is restricted over the range
  $\dist(\ka_1,\ka_2) \approx 2^{-l_1}$ or $\dist(\ka_1,\ka_2) \ls
  2^{-l_1}$ in the case $|l_1-\min(k_1,k_2)| \leq 2$.  The result for
  the second sum follows from the \eqref{basnul}. The first sum, 
  with respect to $l_1$ (the one with respect to $l_2$
  is redundant), is performed using the factor of $2^{-l_1}$.
  
  The proof of \eqref{bil3} is entirely similar, expect that in the
  decomposition above one imposes the range $l \leq l_1,l_2 \leq \min(k_1,k_2)$ 
  on the first sum and picks up
  the additional factor of $2^{-l}$.
  
  Finally, we turn to the proof of \eqref{bil4}. Fix $l_1,l_2$ with
  $|l_1-l_2| \leq 2$, $1 \leq l_1 \leq k_1$, $\ka_1 \in
  \mathcal{K}_{l_1}, \ka_2 \in \mathcal{K}_{l_2}$ with $\dist(\ka_1,
  \ka_2) \approx 2^{-l_1}$ or $\dist( \ka_1, \ka_2) \ls 2^{-l_1}$ in
  the case $|l_1-k_1| \leq 2$. The proof of \eqref{bas1a} yields
  \begin{align*}
    &\sum_{\ka_1 \in \mathcal{K}_{l_1},\ka_2 \in \mathcal{K}_{l_2}} \| \tilde P_{\ka_1}u_1 \overline{\tilde P_{\ka_2} u_2} \|_{L^2_t
      L^q_x}\\
    \ls{} & 2^{(\frac12 - \frac1q)(3k_2-2l_1)}  \sum_{\ka_1 \in
      \mathcal{K}_{l_1},\ka_2 \in \mathcal{K}_{l_2}} \| \tilde P_{\ka_1}u_1 \overline{\tilde P_{\ka_2} u_2}  \|_{L^2_t L^2_x} \\
    \ls{} & 2^{k_1} 2^{(\frac12 - \frac1q)(3k_2-2l_1)} \| u_1 \|_{S^+_{k_1}} \| u_2
    \|_{S_{k_2}^{+,w}},
  \end{align*}
  where the sum is restricted to caps satisfying $\dist(\ka_1, \ka_2)
  \approx 2^{-l_1}$ or $\dist( \ka_1, \ka_2) \ls 2^{-l_1}$ in the case
  $|l_1-k_1| \leq 2$. Summing this inequality with respect to
  $l_1,l_2$ gives \eqref{bil4}.
\end{proof}

\section{The Dirac nonlinearity}\label{sect:dirac}
In this section we use the theory developed in the previous section to
prove the global well-posedness of the Dirac equation with initial
data in $H^1(\R^3)$. Throughout this section we abuse
notation and set $S_{99}^\pm:=S^\pm_{\leq 99}$, redefine $P_{99}:=P_{\leq 99}$,
$\tilde P_{99}:=\tilde P_{\leq 99}$,
and thus by saying that a function is localized at frequency $2^{99}$
we mean that it is localized at frequency $\leq 2^{99}$.  

The main result of this section is the following
\begin{thm} \label{thm:thnon} Choose $s_1,s_2,s_3,s_4 \in \{ +,
  -\}$. Then, for all $\psi_k \in S^{s_k,1}$ satisfying
  $\psi_k=\Pi_{s_k}(D) \psi_k$ for $k=1,2,3$, we have
  \begin{equation} \label{cun} \| \Pi_{s_4}(D) ( \la \psi_1, \beta
    \psi_2 \ra \beta \psi_3) \|_{N^{s_4,1}} \ls \|
    \psi_1\|_{S^{s_1,1}} \| \psi_2 \|_{S^{s_2,1}} \| \psi_3
    \|_{S^{s_3,1}}.
  \end{equation}
\end{thm}
The rest of this section is devoted to the proof of Theorem
\ref{thm:thnon} and the proof of our main result Theorem
\ref{thm:main}.  The estimate \eqref{cun} will be derived from similar
estimates for frequency localized functions. Our aim will be to
identify a function $G(k_1,k_2,k_3,k_4) : \N^4_{\geq 99} \rightarrow (0,\infty)$
such that
\begin{equation} \label{G} \sum_{k_1,k_2,k_3,k_4 \in \N_{\geq 99}}
  G(k_1,k_2,k_3,k_4) a_{k_1} b_{k_2} c_{k_3} d_{k_4} \ls \| a \|_{l^2}
  \| b \|_{l^2} \| c \|_{l^2} \| d \|_{l^2}
\end{equation}
for all sequences $a=(a_j)_{j \in \N_{\geq 99}}$, etc, in $l^2$. 
Here $\N_{\geq 99}=\{ n \in \N | n \geq 99 \}$.
We set $\mathbf{k}=(k_1,k_2,k_3,k_4)$.

With these notations, the result of Theorem \ref{thm:thnon} follows
from
\begin{pro} \label{pro:thnon} There exists a function $G$ satisfying
  \eqref{G} such that if $\psi_j$ are localized at frequency $2^{k_j}$, $k_j \geq 99$
  and $\psi_j = \Pi_{s_j}(D) \psi_j$ for $j=1,\ldots,4$, then the
  following holds true
  \begin{equation} \label{cunn} 2^{k_4} \| P_{k_4} \Pi_{s_4}(D) ( \la
    \psi_1, \beta \psi_2 \ra \beta \psi_3) \|_{N^{s_4}_{k_4}} \ls
    G(\mathbf{k}) \prod_{j=1}^3 2^{k_j} \| \psi_j \|_{S^{s_j}_{k_j}},
  \end{equation}
  for any choice of sign $s_1,s_2,s_3,s_4 \in \{ +, -\}$.
\end{pro}
We break this down into two building blocks:

\begin{lem}\label{lem:cunn1}
  Under the assumptions of Proposition \ref{pro:thnon} the following
  estimate holds true:
  \begin{equation} \label{cunn1} 2^{\frac34 k_4} \| P_{k_4} ( \la
    \psi_1, \beta \psi_2 \ra \beta \psi_3) \|_{L^\frac{4}{3}_t L^2_x}
    \ls G(\mathbf{k}) \prod_{j=1}^3 2^{k_j} \| \psi_j
    \|_{S^{s_j}_{k_j}}.
  \end{equation}
\end{lem}

\begin{lem}\label{lem:cunn23}
  Under the assumptions of Proposition \ref{pro:thnon} the following
  estimates hold true:
  \begin{equation} \label{cunn2}
    \begin{split}
      &\Big| \int  \la  \psi_1, \beta \psi_2 \ra \cdot \la \psi_3, \beta \psi_4 \ra dx dt\Big|\\
      \ls{}& G(\mathbf{k})\prod_{j=1}^3 2^{k_j} \| \psi_j
      \|_{S^{s_j}_{k_j}} \cdot 2^{-k_4} \| \psi_4
      \|_{S^{s_4,w}_{k_4}},
    \end{split}
  \end{equation}
  and
  \begin{equation} \label{cunn3}
    \begin{split}
      &\Big| \int  \la  \psi_1, \beta \psi_2 \ra \cdot \la \psi_3, \frac{\psi_4}{\la D \ra} \ra dx dt\Big|\\
      \ls{}& G(\mathbf{k})\prod_{j=1}^3 2^{k_j} \| \psi_j
      \|_{S^{s_j}_{k_j}} \cdot 2^{-k_4} \| \psi_4
      \|_{S^{s_4,w}_{k_4}}.
    \end{split}
  \end{equation}
\end{lem}

Before we provide proofs of Lemma \ref{lem:cunn1} and Lemma
\ref{lem:cunn23}, we show how these imply Proposition \ref{pro:thnon}.
\begin{proof}[Proof of Prop. \ref{pro:thnon}]
  Given the structure of the $N^{s_4}_{k_4}$, \eqref{cunn1} is simply
  the $L^\frac43_t L^2_x$ part of \eqref{cunn}. We owe an explanation
  for why \eqref{cunn2} and \eqref{cunn3} imply the atomic part of
  \eqref{cunn}. The nonlinearity \[\mathcal{N}= P_{k_4} \Pi_{s_4}(D) (
  \la \psi_1, \beta \psi_2 \ra \beta \psi_3)\] satisfies
  $\mathcal{N}=\Pi_{s_4}(D) \mathcal{N}$ and needs to be estimated in
  $N^{s_4}_{k_4}$. Using \eqref{dual}, it is enough to test
  $\Pi_{s_4}(D) \mathcal{N}$ against $\psi_4 \in S^{-s_4,w}_{k_4}$ and
  to prove the bound
  \begin{equation} \label{dualin} \int \la \Pi_{s_4}(D) \mathcal{N},
    \psi_4 \ra dx dt \ls G(\mathbf{k}) \prod_{j=1}^3 2^{k_j} \| \psi_j
    \|_{S^{s_j}_{k_j}} \cdot 2^{-k_4} \| \psi_4 \|_{S^{-s_4,w}_{k_4}}.
  \end{equation}
  We have
  \begin{align*}
    \int \la \Pi_{s_4}(D) \mathcal{N}, \psi_4 \ra dx
    & = \int \la \hat{\mathcal{N}}(\xi), \Pi_{s_4}(\xi)  \hat{\psi}_4 (-\xi) \ra d \xi \\
    & = \int \la \hat{\mathcal{N}}(\xi), (\Pi_{-s_4}(-\xi) - s_4 \frac{\beta}{\la \xi \ra}) \hat{\psi}_4 (-\xi) \ra d \xi \\
    & = \int \la \mathcal{N}, \Pi_{-s_4}(D) \psi_4 \ra dx - s_4 \int
    \la \mathcal{N}, \frac\beta{\la D \ra} \psi_4 \ra dx
  \end{align*}
  The contribution of the first term to \eqref{dualin} is
  \[
  \begin{split}
    \int \la \mathcal{N}, \Pi_{-s_4}(D) \psi_4 \ra dx dt
    & = \int \la \la  \psi_1, \beta \psi_2 \ra \beta \psi_3, \Pi_{-s_4}(D) \psi_4 \ra dx dt \\
    & = \int \la  \psi_1, \beta \psi_2 \ra \la  \beta \psi_3, \Pi_{-s_4}(D) \psi_4 \ra dx dt \\
    & = \int \la \psi_1, \beta \psi_2 \ra \la \psi_3, \beta
    \Pi_{-s_4}(D) \psi_4 \ra dx dt.
  \end{split}
  \]
  By splitting each $\psi_j = \Pi_{+}(D) \psi_j + \Pi_{-}(D) \psi_j$,
  its contribution to \eqref{dualin} follows from \eqref{cunn2}. The
  reason why the contribution of the second term above to
  \eqref{dualin} is provided by \eqref{cunn3} is similar.
\end{proof}

\begin{proof}[Proof of Lemma \ref{lem:cunn1}] We prove the result by
  using Strichartz type estimates only, thus we can drop all the $\pm$
  and simply use scalar functions $u_j$ localized at frequency
  $2^{k_j}$ instead. The argument is symmetric with respect to
  $k_1,k_2,k_3$, hence we can simply assume that $k_1 \leq k_2 \leq
  k_3$. Then, the l.h.s.\ of \eqref{cunn1} vanishes unless $k_4\leq k_3+10$,
  and by using \eqref{bil4}, \eqref{Str4} and \eqref{Strlow} we obtain
  \begin{align*}
    \| u_1 u_2 u_3 \|_{L^\frac43_t L^2_x} \ls& \| u_1 \|_{L^4_t
      L^{24}_x} \| u_2 u_3 \|_{L^2_t L^{\frac{24}{11}}_x} \\
    \ls{}& 2^{\frac{27}{24}k_1+k_2+\frac{1}{8}k_3} \| u_1 \|_{S_{k_1}}
    \| u_2 \|_{S_{k_2}} \| u_3 \|_{S_{k_3}}.
  \end{align*}
  From this we obtain
  \begin{align*}
    &2^{\frac34 k_4} \| u_1 u_2 u_3 \|_{L^\frac43_t L^2_x} \ls
    2^{\frac{k_1-7k_3+6k_4}8} 2^{k_1} \| u_1 \|_{S_{k_1}} 2^{k_2} \|
    u_2 \|_{S_{k_2}} 2^{k_3} \| u_3 \|_{S_{k_3}}
  \end{align*}
  from which \eqref{cunn1} follows, because the value of
  $G(\mathbf{k})=2^{\frac{k_1-7k_3+6k_4}8}$ is acceptable for $k_4\leq
  k_3+10$.
\end{proof}

It remains to prove Lemma \ref{lem:cunn23}. Before we start to do so,
we analyze the modulation of a product of two waves. We consider two functions
$\psi_1, \psi_2 \in S^{+}$ where their native modulation is
with respect to the quantity $|\tau - \la \xi \ra|$. However, for
$\la \psi_1 , \beta \psi_2 \ra$ we quantify the output modulation
with respect to $||\tau|-\la \xi \ra|$. The following lemma contains the modulation
localization claim which will be used several times in the argument.

\begin{lem} \label{lem:mod} Let $k,k_1k_2\geq
  100$ and $l\prec \min(k_1,k_2)$, and let
  $\kappa_1,\kappa_2 \in \mathcal{K}_{l}$, with
  $\dist(\kappa_1,\kappa_2)\approx 2^{-l}$, and assume that
  $u_j=\tilde{P}_{k_j,\kappa_j}\tilde{Q}^+_{\prec m}u_j$, where
  \[m=k_1+k_2-k-2l.\]
  Then, \[\widehat{P_k(u_1\overline{u_2})}(\tau,\xi)=0 \text{ unless }
  ||\tau|-\la \xi \ra |\approx 2^{m}.\]
\end{lem}
\begin{proof}
  Since the modulation of the inputs are much less than the claimed
  modulation of the output it is enough to prove the argument for free
  solutions. Let $(\xi_1, \la \xi_1 \ra)$ be in the support of
  $\hat{u}_1$ and $(-\xi_2, -\la \xi_2 \ra)$ be in the support of
  $\overline{\hat{u}_2}$. Then, the angle between $\xi_1$ and $\xi_2$
  is $\approx 2^{-l}$. Let $\xi=\xi_1-\xi_2$ be of size $2^{k}$ and
  $\tau=\la \xi_1 \ra -\la \xi_2\ra$. Our aim is to prove that
  \[|\la \xi_1-\xi_2\ra-|\la \xi_1 \ra -\la \xi_2\ra||\approx 2^{m}.\]
  The claim follows from
  \begin{align*}
    &\la \xi_1-\xi_2\ra-|\la \xi_1 \ra -\la \xi_2\ra|=\frac{\la \xi_1-\xi_2\ra^2-(\la \xi_1\ra  -\la \xi_2\ra)^2}{\la \xi_1-\xi_2\ra+|\la \xi_1\ra  -\la \xi_2 \ra|}\\
    =& \frac{2 |\xi_1||\xi_2|(1-\cos(\angle(\xi_1,\xi_2)))}{\la
      \xi_1-\xi_2\ra+|\la \xi_1\ra -\la \xi_2 \ra|}
    +\mathcal{O}(2^{-\min(k,k_1,k_2)})
    \\
    \approx &2^{k_1+k_2-k}\angle(\xi_1,\xi_2)^2
    +\mathcal{O}(2^{-\min(k,k_1,k_2)}).
  \end{align*}
  because by assumption we have $2^{k_1+k_2-k-2l}\gg
  2^{-\min(k,k_1,k_2)}$.
\end{proof}

\begin{proof}[Proof of Lemma \ref{lem:cunn23}]
  It will be obvious from the proof of \eqref{cunn2} that the same
  argument works for \eqref{cunn3} as well. The basic idea in
  \eqref{cunn3} is that the null condition is missing in the term $\la
  \psi_3, \frac{1}{\la D \ra} \psi_4 \ra$. On the other hand the
  factor $\frac{1}{\la D \ra} $ brings a gain of $2^{-k_4}$ in all
  estimates which is better than all gains from exploiting the null
  condition in $\la \psi_3, \beta \psi_4 \ra$.

  Given the choices of sign in \eqref{cunn2} there are a total of $16$
  cases. The first major block in the proof is the use of the results
  in Proposition \ref{Pbil} which are symmetric with respect to the
  choice of $\pm$. The second building block employs frequency and
  modulation localization, Strichartz and Sobolev estimates and it
  works again the same way for different choices of $\pm$ in the
  estimate above. This is why we choose to prove the above estimate
  for the $+$ choice in all terms. It will become evident from the
  argument that the same reasoning will work in all other cases. Thus
  we can drop all the $\pm$ and simply consider $\psi_j \in S^+_{k_j}$
  and write $S_{k_j}=S^+_{k_j}$ instead.

  For brevity, we denote the l.h.s.\ of \eqref{cunn2} as
  \begin{equation*}
    I:=\Big| \int \la \psi_1, \beta \psi_2 \ra \cdot \la \psi_3, \beta
    \psi_4 \ra dx dt\Big|
  \end{equation*}
  and the standard factor on the r.h.s.\ as
  \begin{equation*}
    J:=\prod_{j=1}^3 2^{k_j} \| \psi_j \|_{S_{k_j}} \cdot 2^{-k_4} \| \psi_4 \|_{S^{w}_{k_4}}.
  \end{equation*}
  Since the expression $I$ computes the zero mode of the product 
  $\la \psi_1, \beta \psi_2 \ra \cdot \la \psi_3, \beta \psi_4 \ra$, it follows that
  $\la \psi_1, \beta \psi_2 \ra$ and $\la \psi_3, \beta \psi_4 \ra$ need to be localized 
  at frequencies and modulations of comparable size, where the modulation is computed 
  with respect to $||\tau|-\la \xi \ra|$. This will be repeatedly used in the argument below along with
  the convention that the modulations of $\psi_k, k=1,\ldots,4$ are with respect to $|\tau-\la \xi \ra|$,
  while the modulations of $\la \psi_1, \beta \psi_2 \ra$ and $\la \psi_3, \beta \psi_4 \ra$ are
  with respect to $||\tau|-\la \xi \ra|$. 
  
  We also agree that by the angle of interaction in, say, $\la \psi_1, \beta \psi_2 \ra$
  we mean  the angle made by the frequencies in the support of $\hat{\psi}_1$
  and $\hat{\psi}_2$, where we consider only the supports that bring nontrivial contributions
  to $I$. 
  
  We organize the argument based on the size of the frequencies.

  \underline{Case 1: $k_4 \leq \min(k_1,k_2,k_3) + 10$.}\\
  Using \eqref{bil2} we obtain the bound \[I \ls 2^{k_4-\max(k_1,k_2)}
  J,\] and since $|\max(k_1,k_2) - \max(k_1,k_2,k_3)| \leq 12$ we obtain
  \eqref{cunn2} in this case.

  \underline{Case 2: there are at exactly two $i \in \{ 1,2,3 \}$ such that $k_4 \leq k_i + 10$.}\\
  {\it Case 2 a)} Assume that $k_3 \geq k_4-10$. Since the argument is
  symmetric in $k_1$ and $k_2$, it is enough to consider the scenario
  $k_1 < k_4-10 \leq k_2$. Note that $|k_2-k_3| \leq 12$. 

  We claim that either the angle of interactions in $\la \psi_3,\beta
  \psi_4 \ra$ is $\ls 2^{\frac{k_1-k_4}8}$ or at least one factor $\psi_j, j=1,..,4$ has
   modulation $\gs 2^{\frac{k_1+3k_4}4}$. To see this, suppose
  that the claim is false. Then, the modulation of $\la \psi_1,
  \beta \psi_2 \ra$ is $\ls 2^{\frac{k_1+3k_4}4}$ while it follows
  from Lemma \ref{lem:mod} that the modulation of $\la
  \psi_3, \beta \psi_4 \ra$ is $\gg 2^{\frac{k_1+3k_4}4}$. This is not
  possible, hence the claim is true. Note that in using Lemma \ref{lem:mod}
  we are assuming that $k_3,k_4 \geq 100$. If this is not the case, that is either $k_3=99$
  or $k_4=99$, the argument in Case 1 can be used to obtain the desired estimate. 
	
  In the first subcase, where the angle of interaction in $\la
  \psi_3,\beta \psi_4 \ra$ is smaller than $2^{\frac{k_1-k_4}8}$, we
  use \eqref{bil3} to obtain $I \ls 2^{\frac{k_1-k_4}8} 2^{k_4-k_2} J$
  and this is fine.

  We now consider the second subcase, in which the modulation of the
  factor $\psi_j$ is $\gs 2^{\frac{k_1+3k_4}4} \gs
  2^{\frac{k_1+k_4}2}$ for some $j\in \{1,2,3,4\}$:

  $j=1$: Since $\psi_1$ has modulation $\gs 2^{\frac{k_1+k_4}2}$, we
  can use \eqref{bil2} to estimate $\| \la \psi_3 , \beta \psi_4 \ra
  \|_{L^2}$ and the Sobolev embedding for $\psi_1$ to obtain
  \begin{align*}
    I \ls{}& \| \psi_1 \|_{L^2_tL^\infty_x} \| \psi_2 \|_{L^\infty_t
      L^2_x} 2^{k_3}\| \psi_3 \|_{S_{k_3}} \| \psi_4 \|_{S^w_{k_4}} \\
    \ls{}& 2^{\frac32 k_1} \| \psi_1 \|_{L^2} \| \psi_2 \|_{L^\infty_t L^2_x} 2^{k_3} \| \psi_3 \|_{S_{k_3}} \| \psi_4 \|_{S^w_{k_4}} \\
    \ls{}& 2^{\frac{k_1-k_4}4} 2^{k_4-k_2} J.
  \end{align*}

  $j=2$: Since $\psi_2$ has modulation $\gs 2^{\frac{k_1+k_4}2}$,
  \eqref{bil2} and Sobolev embedding for $\psi_1$ yields
  \begin{align*}
    I \ls{}& \| \psi_1 \|_{L^\infty} \| \psi_2 \|_{L^2} 2^{k_3} \| \psi_3 \|_{S_{k_3}} \| \psi_4 \|_{S^w_{k_4}} \\
    \ls{}& 2^{\frac32 k_1} \| \psi_1 \|_{L^\infty_t L^2_x}
    2^{-\frac{k_1+k_4}4} \| \psi_2 \|_{S_{k_2}}
    2^{k_3} \| \psi_3 \|_{S_{k_3}} \| \psi_4 \|_{S^w_{k_4}} \\
    \ls{}& 2^{\frac{k_1-k_4}4} 2^{k_4-k_2} J.
  \end{align*}
 
  $j=3$: Because $\psi_3$ has modulation $\gs 2^{\frac{k_1+3k_4}4}$,
  we employ \eqref{Str4} to estimate $\| \la \psi_1 , \beta \psi_2
  \ra \|_{L^2}$ and the Sobolev embedding for $\psi_4$ to obtain
  \begin{align*}
    I  \ls{}& \| \psi_1 \|_{L^4_t L^\infty_x} \| \psi_2 \|_{L^4_t L^6_x} \| \psi_3 \|_{L^2} \| \psi_4 \|_{L^\infty_t L^3_x} \\
    \ls{}& 2^{\frac54 k_1} \| \psi_1 \|_{S_{k_1}} 2^{\frac34 k_2} \|
    \psi_2 \|_{S_{k_2}}
    2^{-\frac{k_1+3k_4}8}  \| \psi_3 \|_{S_{k_3}}  2^{\frac{k_4}2} \| \psi_4 \|_{L^\infty_t L^2_x} \\
    \ls{}& 2^{\frac{k_1-k_4}8} 2^{k_4-k_3} J.
  \end{align*}
  
  $j=4$: Here, $\psi_4$ has modulation $\gs 2^{\frac{k_1+k_4}2}$ and
  we use \eqref{Str4} and the Sobolev embedding for $\psi_1$ to obtain
  \begin{align*}
    I \ls{}& \| \psi_1 \|_{L^\infty_t L_x^{12}} \| \psi_2 \|_{L^4_t
      L_x^{\frac{24}5}} \| \psi_3 \|_{L^4_t L_x^{\frac{24}5}}
    \| \psi_4 \|_{L^2} \\
    \ls{}& 2^{\frac54 k_1} \| \psi_1 \|_{L^\infty_t L^2_x} 2^{\frac58
      k_2} \| \psi_2 \|_{S_{k_2}}
    2^{\frac58 k_3} \| \psi_3 \|_{S_{k_3}} 2^{-\frac{k_1+3k_4}8} \| \psi_4 \|_{S_{k_4}^w} \\
    \ls{}& 2^{\frac{k_1-k_4}8} 2^{\frac{k_4-k_2}2} J.
  \end{align*}

  {\it Case 2 b)} Assume now that $k_3 \leq k_4+10$, hence $k_1,k_2
  \geq k_4+10$ and $|k_1-k_2| \leq 12$. Since $\la \psi_1,\beta
  \psi_2 \ra$ is localized at frequency $\approx 2^{k_4}$, the angle of interaction in
  $\la \psi_1,\beta \psi_2 \ra$ is $\ls 2^{k_4-k_2}$. Moreover, we
  claim that either the angle of interactions in $\la \psi_1,\beta
  \psi_2 \ra$ is $\ls 2^{\frac{k_3-k_4}8} 2^{k_4-k_2}$ or at least one
  factor $\psi_j, j=1,..,4$ has modulation $\gs 2^{\frac{k_3+3k_4}4}$. Indeed, if
  the claim is false, it follows from Lemma \ref{lem:mod} that the
  modulation of $\la \psi_1, \beta \psi_2 \ra$ is $\gg
  2^{\frac{k_3+3k_4}4}$ while the modulation of $\la
  \psi_3, \beta \psi_4 \ra$ is $\ll 2^{\frac{k_3+3k_4}4}$. This is not
  possible, hence the claim is true. Note that in using Lemma \ref{lem:mod}
  we are assuming that $k_1,k_2 \geq 100$. If this is not the case, that is either $k_1=99$
  or $k_2=99$, the argument in Case 1 can be used to obtain the desired estimate. 

  In the first subcase the angle of interaction in $\la \psi_1,\beta
  \psi_2 \ra$ is smaller than $2^{\frac{k_3-k_4}8} 2^{k_4-k_2}$. Then,
  we use \eqref{bil3} to obtain $I \ls 2^{\frac{k_3-k_4}8}
  2^{2(k_4-k_2)} J$ which is acceptable.

  In the second subcase, where at least one modulation is high, we
  proceed in a similar manner to Case 2b) above. In fact the estimates
  bring improved factors if one takes into account that the angle of
  interaction in $\la \psi_1,\beta \psi_2 \ra$ is $\ls
  2^{k_4-k_2}$. The details are left to the reader.
 
  \underline{Case 3: $| k_2 -k_4 | \leq 2$ and $k_1,k_3 \leq k_4
    -10$.}  Without restricting the generality of the argument, we may
  assume that $k_1 \leq k_3$.
 
  We claim that either the angle of interaction in $ \la \psi_3, \beta
  \psi_4 \ra$ is $\ls 2^{\frac{k_1-k_3}{16}}$ or one factor $\psi_j, j=1,..,4$ has
  modulation $\gs 2^{\frac{k_1+7k_3}8}$. Indeed, if all modulations of the functions involved are $\ll
  2^{\frac{k_1+7k_3}8}$, then $ \la \psi_1, \beta \psi_2
  \ra$ is localized at modulation $\ls 2^{\frac{k_1+7k_3}8}$.  This
  forces $ \la \psi_3, \beta \psi_4 \ra$ to be localized at modulation
  $\ls 2^{\frac{k_1+7k_3}8}$, hence the angle of interaction is $\ls
  2^{\frac{k_1-k_3}{16}}$  by Lemma \ref{lem:mod}. Note that in using Lemma \ref{lem:mod}
  we are assuming that $k_3,k_4 \geq 100$. If this is not the case, that is either $k_3=99$
  or $k_4=99$, the argument in Case 1 can be used to obtain the desired estimate.

  In the first subcase, when the angle of interaction in $ \la \psi_3,
  \beta \psi_4 \ra$ is $\ls 2^{\frac{k_1-k_3}{16}}$, we use
  \eqref{bil3} to obtain
  \[
  I \ls 2^{\frac{k_1-k_3}{16}} 2^{k_1} \| \psi_1 \|_{S_{k_1}} \|
  \psi_2 \|_{S_{k_2}} 2^{k_3} \| \psi_3 \|_{S_{k_3}} \| \psi_4
  \|_{S_{k_4}^w} \ls 2^{\frac{k_1-k_3}{16}} J.
  \]

  Next, we consider the second subcase when the factor $\psi_j$ has
  modulation $\gs 2^{\frac{k_1+7k_3}8} \gs 2^{\frac{k_1+3k_3}4} $ for
  some $j \in \{1,2,3,4\}$:
 
  $j=1$: The modulation of $\psi_1$ is $\gs 2^{\frac{k_1+3k_3}4}$, so
  we use Sobolev embedding for $\psi_1$ and \eqref{bil2} for $\la
  \psi_3, \beta \psi_4 \ra$ to obtain
  \begin{align*}
    I \ls{}& \| \psi_1 \|_{L^2_t L^\infty_x} \| \psi_2 \|_{L^\infty_t
      L^2_x} 2^{k_3}
    \| \psi_3 \|_{S_{k_3}} \| \psi_4 \|_{S_{k_4}^w} \\
    \ls{}& 2^{\frac32 k_1} \| \psi_1 \|_{L^2} \| \psi_2 \|_{S_{k_2}}
    2^{k_3}
    \| \psi_3 \|_{S_{k_3}} \| \psi_4 \|_{S_{k_4}^w} \\
    \ls{}& 2^{\frac32 k_1} 2^{-\frac{k_1+3k_3}8} \| \psi_1
    \|_{S_{k_1}} \| \psi_2 \|_{S_{k_2}} 2^{k_3}
    \| \psi_3 \|_{S_{k_3}} \| \psi_4 \|_{S_{k_4}^w} \\
    \ls{}& 2^{\frac{k_1-k_3}8} J.
  \end{align*}

  $j=2$: Here, the modulation of $\psi_2$ is $\gs
  2^{\frac{k_1+3k_3}4}$ and we proceed as above to obtain
  \begin{align*}
    I \ls{}& \| \psi_1 \|_{L^\infty} \| \psi_2 \|_{L^2} 2^{k_3}
    \| \psi_3 \|_{S_{k_3}} \| \psi_4 \|_{S_{k_4}^w} \\
    \ls{}& 2^{\frac32 k_1} \| \psi_1 \|_{L^\infty_t L^2_x}
    2^{-\frac{k_1+3k_3}8} \| \psi_2 \|_{S_{k_2}} 2^{k_3}
    \| \psi_3 \|_{S_{k_3}} \| \psi_4 \|_{S_{k_4}^w} \\
    \ls{}& 2^{\frac{k_1-k_3}8} J.
  \end{align*}

  $j=3$: The modulation of $\psi_3$ is $\gs 2^{\frac{k_1+7k_3}8}$, so
  we use \eqref{Str4} and the Sobolev embedding for $\psi_3$ to obtain
  \begin{align*}
    I \ls{}& \| \psi_1 \|_{L^4_t L^\infty_x} \| \psi_2 \|_{L^\infty_t L^2_x}  \| \psi_3 \|_{L^\frac43_t L^\infty_x} \| \psi_4 \|_{L^\infty_t L^2_x} \\
    \ls{}& 2^{\frac54 k_1} \| \psi_1 \|_{S_{k_1}} \| \psi_2
    \|_{S_{k_2}} 2^{\frac32 k_3} \| \psi_3 \|_{L^\frac43_t L^2_x}
    \| \psi_4 \|_{S^{w}_{k_4}} \\
    \ls{}&  2^{\frac54 k_1}  \| \psi_1 \|_{S_{k_1}} \| \psi_2 \|_{S_{k_2}} 2^{\frac32 k_3} 2^{\frac{k_3}4} 2^{-\frac{k_1+7k_3}8} \| \psi_3 \|_{S_{k_3}}   \| \psi_4 \|_{S^{w}_{k_4}} \\
    \ls{}& 2^{\frac{k_1-k_3}8} J.
  \end{align*}

  $j=4$: Since the modulation of $\psi_4$ is $\gs
  2^{\frac{k_1+3k_3}4}$, we use \eqref{Str4} to obtain
  \begin{align*}
    I \ls{}& \| \psi_1 \|_{L^4_t L^\infty_x} \| \psi_2 \|_{L^\infty_t L^2_x}  \| \psi_3 \|_{L^4_t L^\infty_x} \| \psi_4 \|_{L^2} \\
    \ls{}& 2^{\frac54 k_1} \| \psi_1 \|_{S_{k_1}} \| \psi_2 \|_{S_{k_2}}  2^{\frac54 k_3} \| \psi_3 \|_{S_{k_3}} 2^{-\frac{k_1+3k_3}8} \| \psi_4 \|_{S^{w}_{k_4}} \\
    \ls{}& 2^{\frac{k_1-k_3}8} J.
  \end{align*}
  
  \underline{Case 4: $|k_3-k_4| \leq 2$ and $k_1,k_2 \leq k_4 -
    10$}. Without loss of generality we assume $k_1 \leq k_2$.

  The key observation is that either the angle of interaction between
  $\psi_3$ and $\psi_4$ is $\ls 2^{\frac{k_1-k_2}{16}} 2^{k_2-k_3}$ or
  at least one factor has modulation $\gs
  2^{\frac{k_1+7k_2}8}$. Indeed, if all modulations are $\ll
  2^{\frac{k_1+7k_2}8}$, then the modulation of $\la
  \psi_1,\beta\psi_2\ra$ is $\ls 2^{\frac{k_1+7k_2}8}$ and Lemma
  \ref{lem:mod} implies the claim. Note that in using Lemma \ref{lem:mod}
  we are assuming that $k_3,k_4 \geq 100$. If this is not the case, that is either $k_3=99$
  or $k_4=99$, the argument in Case 1 can be used to obtain the desired estimate.

  In the first subcase, when the angle of interaction between $\psi_3$
  and $\psi_4$ is $\ls 2^{\frac{k_1-k_2}{16}} 2^{k_2-k_3}$, we use
  \eqref{bil3} to obtain
  \[
  I \ls 2^{\frac{k_1-k_2}{16}} J.
  \]

  In the second subcase, $\psi_j$ has modulation $\gs
  2^{\frac{k_1+7k_2}8}$ for some $j \in \{1,2,3,4\}$. Since the output
  of $\la \psi_3 , \beta \psi_4 \ra$ is localized at frequency $\ls
  2^{k_2}$ it follows that the angle of interaction is $\ls
  2^{k_2-k_3}$. This will be used in the following case-by-case
  analysis:

  $j=1$: The modulation of $\psi_1$ is $\gs 2^{\frac{k_1+7k_3}{8}}$,
  so we use Sobolev embedding for $\psi_1$ and \eqref{bil3} for $\la
  \psi_3, \beta \psi_4 \ra$ to obtain
  \begin{align*}
    I \ls{}& \| \psi_1 \|_{L^2_t L^\infty_x} \| \psi_2 \|_{L^\infty_t
      L^2_x} \| \la \psi_3,\beta\psi_4 \ra \|_{L^2} \\
    \ls{}& 2^{\frac32 k_1} \| \psi_1 \|_{L^2} \| \psi_2 \|_{S_{k_2}}
    2^{k_3}2^{k_2-k_3} \| \psi_3 \|_{S_{k_3}} \| \psi_4 \|_{S_{k_4}^w} \\
    \ls{}& 2^{\frac32 k_1} 2^{-\frac{k_1+7k_2}{16}} \| \psi_1
    \|_{S_{k_1}}
    \| \psi_2 \|_{S_{k_2}} 2^{k_2} \| \psi_3 \|_{S_{k_3}} \| \psi_4 \|_{S_{k_4}^w} \\
    \ls{}& 2^{-\frac{7}{16}(k_2-k_1)} J.
  \end{align*}

  $j=2$: The modulation of $\psi_2$ is $\gs 2^{\frac{k_1+7k_3}8}$, so
  we use Sobolev embedding for $\psi_1$ and \eqref{bil3} for $\la
  \psi_3, \beta \psi_4 \ra$ to obtain
  \begin{align*}
    I \ls{}& \| \psi_1 \|_{L^\infty} \| \psi_2 \|_{L^2} \| \la \psi_3,\beta\psi_4 \ra \|_{L^2} \\
    \ls{}& 2^{\frac32 k_1} \| \psi_1 \|_{L^\infty_t L^2_x}
    2^{-\frac{k_1+7k_2}{16}}\| \psi_2 \|_{S_{k_2}}
    2^{k_3}2^{k_2-k_3} \| \psi_3 \|_{S_{k_3}} \| \psi_4 \|_{S_{k_4}^w} \\
    \ls{}& 2^{-\frac{7}{16}(k_2-k_1)} J.
  \end{align*}

  $j=3$: The modulation of $\psi_3$ is $\gs 2^{\frac{k_1+7k_3}8}$, so
  we use \eqref{Str4} for $\psi_1$ and $\psi_2$ and obtain
  \begin{align*}
    I \ls{}& \| \psi_1 \|_{L^4_tL^\infty_x} \| \psi_2 \|_{L^4_tL^\infty_x} \|\psi_3\|_{L^2}\|\psi_4\|_{L^\infty_tL^2_x} \\
    \ls{}& 2^{\frac54 k_1} \| \psi_1 \|_{L^\infty_t L^2_x} 2^{\frac54
      k_2}\| \psi_2 \|_{S_{k_2}}
    2^{-\frac{k_1+7k_2}{16}}\| \psi_3 \|_{S_{k_3}} \| \psi_4 \|_{S_{k_4}^w} \\
    \ls{}& 2^{-\frac{3}{16}(k_2-k_1)} J.
  \end{align*}

  $j=4$: The modulation of $\psi_4$ is $\gs 2^{\frac{k_1+7k_3}8}$, so
  after exchanging the roles of $\psi_3$ and $\psi_4$ the same
  argument as in case $j=3$ applies.
\end{proof}

Based on Theorem \ref{thm:thnon} we can now prove Theorem
\ref{thm:main} concerning the global well-posedness and scattering of
the cubic Dirac equation for small data.
\begin{proof}[Proof of Theorem \ref{thm:main}]
  In Section \ref{sect:setupD} we reduced the study of the cubic Dirac
  equation to the study of the system \eqref{CDsys}.  In the
  nonlinearity of \eqref{CDsys} we split the functions into
  $\psi=\psi_+ + \psi_-$ where $\psi_\pm = \Pi_\pm \psi$ and note that
  $\psi_\pm = \Pi_\pm \psi_\pm$. Using the nonlinear estimate in
  Theorem \ref{thm:thnon} and the linear estimates in Corollary
  \ref{corlin}, a standard fixed point argument in a small ball in the
  space $S^{+,1} \times S^{-,1}$ gives global existence, uniqueness
  and Lipschitz continuity of the flow map for small initial data
  $(\psi_+(0),\psi_-(0))\in H^1(\R^3)\times H^1(\R^3)$. Concerning
  scattering, we simply argue as follows: Let $\psi \in S^1$ be a solution to
  the cubic Dirac equation constructed above, where $S^1$ is the space
  of all $\psi$ such that $\Pi_\pm\psi \in S^{\pm,1}$. Choose initial
  data $\psi_n(0)\in H^2(\R^3)$ with
  $\|\psi_n(0)-\psi(0)\|_{H^1(\R^3)}\to 0$ as $n \to \infty$, and
  denote the corresponding solutions in $S^1$ by $\psi_n$. By
  continuity we have $\|\psi_n- \psi\|_{S^1}\to 0$ as $n \to
  \infty$. From the scattering result in \cite[Theorem 1]{MNO} we
  infer that there exist solutions to the linear Dirac equation
  $\varrho_n^{\pm \infty}$ such that $\|\psi_n(t)- \varrho_n^{\pm \infty}(t)\|_{H^2}\to 0$ as
  $t \to \pm\infty$. Let $\eps>0$. There exists $n_0$, such that for
  $n,m\geq n_0$ and sufficiently large $\pm t$ we have
  \begin{align*}
    &\|\varrho_n^{\pm \infty}(0) -\varrho_m^{\pm \infty}(0)\|_{H^1}=\|\varrho_n^{\pm \infty}(t) -\varrho_m^{\pm \infty}(t)\|_{H^1}\\
    \leq{}& \|\varrho_n^{\pm \infty}(t) -\psi_n(t)\|_{H^1} +\|\psi_n(t)
    -\psi_m(t)\|_{H^1}+\|\psi_m(t)
    -\varrho_m^{\pm \infty}(t)\|_{H^1}<\eps,
  \end{align*}
  hence the Cauchy-sequence $\varrho^{\pm \infty}(0)$ converges to some $\varrho^{\pm \infty}
  \in H^1(\R^3)$. Let $\eps>0$. Then, $n$ can be chosen sufficiently
  large such that for the corresponding solution $\varrho^{\pm \infty}$ to the
  linear Dirac equation with $\varrho^{\pm \infty}(0)=\varrho^{\pm \infty}$ it follows that
  \begin{align*}
    &\limsup_{t \to \pm \infty} \|\psi(t)-\varrho^{\pm \infty}(t)\|_{H^1}
    \leq \sup_{t \in \R}\|\psi(t)-\psi_n(t)\|_{H^1}\\
    &\qquad +\lim_{t \to \pm \infty}\|\psi_n(t)
    -\varrho_n^{\pm \infty}(t)\|_{H^1}+\sup_{t \in \R}\|\varrho_n^{\pm \infty}(t)
    -\varrho^{\pm \infty}(t)\|_{H^1}<\eps,
  \end{align*}
  which proves the scattering claim.
\end{proof}

\appendix
\section{Proofs of the decay estimates}\label{app:decay}
Here, we provide proofs of the well-known decay
estimates in Section \ref{sect:le}, which clearly reveal the frequency
dependence and which are self-contained in the important case $k\geq
1$. We do not claim originality here, compare e.g.\ \cite[Section
2.5]{NaSch}.
\subsection{Proof of Lemma \ref{lem:decay}
  i)}\label{app:l1}
By recaling it suffices to prove the estimate for $k \in \Z$, $k \leq
1$.  Let $\zeta\in C^\infty_c(\R^3)$ be a nonnegative, radial function
with $\zeta(\xi)=1$ for $|\xi|\leq 2^{4}$.  We identify the
oscillatory integral
\[I(t,x)=\int_{\R^3} e^{i(x,t)\cdot(\xi,\la \xi
  \ra)}\zeta(\xi)\,d\xi \] as the (inverse) Fourier transform of the
surface measure of $ \{(\tau,\xi)\in \R^4: \tau = \la \xi \ra\} $
which is induced by $(1+\frac{|\xi|^2}{\la \xi
  \ra^2})^{-\frac12}\zeta(\xi) d\xi$.  In the support of $\zeta$ the
above surface has non-vanishing principal curvatures, and the
classical result on Fourier transforms of surface carried measures
\cite[p. 348, Theorem 1]{St} implies
\[
|I(t,x)| \ls (1+|(t,x)|)^{-\frac32}.
\]
With $f_k(\xi): = \tilde{\chi}_k^2(\xi)$, it holds that
$\check{f_k}(x)=2^{3k}\check{f_1}(2^k x)$, which shows
$\|\check{f_k}\|_{L^1(\R^3)}=\|\check{f_1}\|_{L^1(\R^3)}$. For $k \leq
1$ we obtain $K_k$ as the (spatial) convolution of $I(t ,\cdot)$ and
$\check{f_k}$, which implies
\[
|K_{k}(t,x)| \ls (1+|(t,x)|)^{-\frac32}
\]
by Young's inequality. Estimate \eqref{eq:smallk} follows in the case
$|(t,x)|>2^{-2k}$. In the remaining case $|(t,x)|\leq 2^{-2k}$ the
estimate \eqref{eq:smallk} is trivial.

\subsection{Proof of Lemma \ref{lem:decay} ii)}\label{app:l2}
Consider
\begin{equation}\label{eq:p}
  P_k(s,y)=\int_{\R^3}e^{iy\cdot\xi}e^{is\la \xi\ra_{k}}\zeta(\xi)\,d\xi.
\end{equation}

We claim that for all $k\in \Z, k \gs 1$ and $s\in \R, y \in \R^3$ the
following estimates hold true:
\begin{align}\label{eq:stat-a}
  |P_k(s,y)|\ls& (1+|(s,y)|)^{-1},\\
  \label{eq:stat-b}
  |P_k(s,y)|\ls& 2^k(1+|(s,y)|)^{-\frac32}.
\end{align}
By rescaling $(\tau,\xi) \rightarrow 2^k(\tau,\xi)$, we have
$K_k(t,x)=2^{3k}P_k(2^k t, 2^k x)$, where
$\zeta(\xi)=\tilde{\chi}^2_1(|\xi|)$. Hence, \eqref{eq:bigk} follows
from \eqref{eq:stat-a} and \eqref{eq:stat-b}, which we will prove
below. Because of the trivial bound
\begin{equation}\label{eq:trivial}
  |P_k(s,y)|\leq \|\zeta\|_{L^1(\R^3)}
\end{equation}
it is enough to treat the case $|(s,y)|\geq 1$.

The function $y \mapsto P_k(s,y)$ is radial, so it suffices to
consider $y=(|y|,0,0)$. By introducing polar coordinates, we obtain
\begin{align}
  P_k(s,(|y|,0,0))=&2\pi \int_0^\infty\int_0^\pi
  e^{ir|y|\cos(\phi)}e^{is\la r \ra_{k}} r^2\zeta(r)\sin(\phi)\, d\phi dr\nonumber\\
  =&2\pi \int_0^\infty\int_{-1}^1 e^{i(r|y|z+s\la r \ra_{k})}
  r^2\zeta(r)\, dz dr \label{eq:p-polar}
\end{align}

\underline{Case $|s|>2^4|y|$:} For a given $z \in [-1,1]$ let
$\phi(r):=r\frac{|y|z}{s}+\la r\ra_k$, such that the phase in
\eqref{eq:p-polar} is given by $s\phi(r)$. Notice that
$\phi'(r)=\frac{|y|z}{s}+\frac{r}{\la r\ra_k}$, so that
$|\phi'(r)|\geq c >0$ and for all $j \geq 2$ it holds
$|\phi^{(j)}(r)|\leq c_j $ for all $r\in \supp(\zeta)$, $z \in
[-1,1]$, $y \in \R^3$, $s\in \R$ and $k \in \N_0$. Multiple
integration by parts with respect to $r$ yields
\begin{equation*}
  |P_k(s,(|y|,0,0))|\leq 4\pi \sup_{z \in [-1,1]}\Big|\int_0^\infty e^{is\phi(r)} r^2\zeta(r)\, dr \Big|\\
  \leq C_N |s|^{-N}
\end{equation*}
for all $N \in \N$ and the claims \eqref{eq:stat-a} and
\eqref{eq:stat-b} follow in this case.

\underline{Case $|s|<2^{-4}|y|$:} The same argument as above applies
if we rewrite the phase function as $|y|\tilde{\phi}(r)$ with
$\tilde{\phi}(r)=rz+\frac{s}{|y|}\la r \ra_{k}$.

\underline{Case $2^{-4}|y|\leq |s|\leq 2^4|y|$:} Integrating
\eqref{eq:p-polar} in $z$ yields
\begin{equation}
  \begin{split}
    P_k(s,(|y|,0,0))={}&\frac{2\pi}{i|y|} \int_0^\infty e^{i(s\la r \ra_{k}+r|y|)} r\zeta(r) \,dr \\
    &{}-\frac{2\pi}{i|y|} \int_0^\infty e^{i(s\la r \ra_{k}-r|y|)}
    r\zeta(r) \,dr.
  \end{split}
  \label{eq:p-polar-int}
\end{equation}
which implies
\begin{equation*}
  |P_k(s,(|y|,0,0))|\leq C |y|^{-1}
\end{equation*}
and the first claim \eqref{eq:stat-a} follows. We can rewrite
\eqref{eq:p-polar-int} as
\begin{equation}\label{eq:p-int-parts}
  \begin{split}
    & P_k(s,(|y|,0,0))=\frac{2\pi}{i|y|}(I(s,y)-\overline{I(-s,y)}),\\
    &\text{ where } I(s,y):= \int_0^\infty e^{i(s\la r \ra_{k}+r|y|)}
    r\zeta(r) \,dr.
  \end{split}
\end{equation}

Let us define the phase function $\varphi(r)=\la r
\ra_{k}+r\frac{|y|}{s}$. We have $\varphi'(r)=\frac{r}{\la r
  \ra_k}+\frac{|y|}{s}$, $\varphi''(r)=\frac{2^{-2k}}{\la r \ra_k^3}$
and for $j \geq 2$ $|\varphi^{(j)}(r)|\leq c_j $ for $r \in
\supp(\zeta)$, $y \in \R^3$, $s\in \R$ and $k \in \N_0$, see
\eqref{eq:jap}.  Notice that $\varphi'$ has a unique zero or does not
vanish. Let us consider the case where $\varphi'(r_0)=0$ for some $r_0
\in \supp(\zeta)$. Then, we have $|\varphi'(r)|\geq c 2^{-2k} |r-r_0|$
in the support of $\zeta$. Let $\delta:=2^k|s|^{-\frac12}$. In case
$\delta\geq 2^{-4}$ the claim \eqref{eq:stat-b} follows from
\eqref{eq:stat-a}, so we may assume that $\delta<2^{-4}$ and we
decompose $\int_0^\infty dr =\int_0^{r_0-\delta} dr+
\int_{r_0-\delta}^{r_0+\delta} dr+ \int_{r_0+\delta}^\infty dr $, in
which case we obtain
\begin{align*}
  \Big|\int_0^{r_0-\delta} e^{i(s\la r \ra_{k}+r|y|)} r\zeta(r) \,dr
  \Big| ={}&|s|^{-1}\Big|\int_0^{r_0-\delta} e^{i(s\la r
    \ra_{k}+r|y|)} \frac{\mathrm
    d}{\mathrm dr}\frac{r\zeta(r)}{\varphi'(r)} \,dr \Big|\\
  \leq{} & |s|^{-1}\int_0^{r_0-\delta}
  \Big|\frac{(r\zeta(r))'}{\varphi'(r)} \Big|\,dr \\
  &{}+|s|^{-1} \int_0^{r_0-\delta}
  r\zeta(r) \big| (\varphi'(r)^{-1})' \big|\,dr\\
  \leq{} & c2^{2k}(\delta |s|)^{-1}+c |s|^{-1}\int_0^{r_0-\delta}
  \big| (\varphi'(r)^{-1})' \big|\,dr\\
  \leq{}& c 2^{2k}(\delta |s|)^{-1}
\end{align*}
where we have used that $(\varphi'(r))^{-1}$ is decreasing in the
domain of integration, which implies that
\begin{equation*}\int_0^{r_0-\delta} \big|
  (\varphi'(r)^{-1})' \big|\,dr\leq
  |\varphi'(r_0-\delta)^{-1}|
  \leq  c2^{2k}(\delta |s|)^{-1}
\end{equation*}
A similar argument applies to the third part, and the second
contribution is trivially bounded by $c\delta$, such that altogether
we obtain
\begin{equation}\label{eq:ibound}
  |I(s,y)|\leq c 2^{2k}(\delta |s|)^{-1}+c\delta\leq c 2^{k}|s|^{-\frac12}.
\end{equation}
The claim \eqref{eq:stat-b} follows by combining \eqref{eq:ibound} and
\eqref{eq:p-int-parts}.  In the remaining case where $\varphi'\not= 0$
in $\supp(\zeta)$, we have $\varphi'(r)\geq c>0$ for all $r\in
\supp(\zeta)$ and we obtain $|I(s,y)|\leq C_N |s|^{-N}$ for every
$N\in \N$ by multiple integration by parts with respect to $r$.

\bibliographystyle{amsplain} \bibliography{cubic-dirac-refs}

\providecommand{\bysame}{\leavevmode\hbox to3em{\hrulefill}\thinspace}
\providecommand{\MR}{\relax\ifhmode\unskip\space\fi MR }
\providecommand{\MRhref}[2]{%
  \href{http://www.ams.org/mathscinet-getitem?mr=#1}{#2}
}
\providecommand{\href}[2]{#2}
\begin{thebibliography}{10}

\bibitem{bikt}
Ioan Bejenaru, Alexandru~D. Ionescu, Carlos~E. Kenig, and Daniel Tataru,
  \emph{Global {S}chr\"odinger maps in dimensions {$d\geq 2$}: small data in
  the critical {S}obolev spaces}, Ann. of Math. (2) \textbf{173} (2011), no.~3,
  1443--1506. \MR{2800718 (2012g:58048)}

\bibitem{B85}
Philip Brenner, \emph{On scattering and everywhere defined scattering operators
  for nonlinear {K}lein-{G}ordon equations}, J. Differential Equations
  \textbf{56} (1985), no.~3, 310--344. \MR{780495 (86f:35155)}

\bibitem{C11}
Timothy Candy, \emph{Global existence for an {$L^2$} critical nonlinear {D}irac
  equation in one dimension}, Adv. Differential Equations \textbf{16} (2011),
  no.~7-8, 643--666. \MR{2829499 (2012f:35452)}

\bibitem{cv}
Thierry Cazenave and Luis V{\'a}zquez, \emph{Existence of localized solutions
  for a classical nonlinear {D}irac field}, Comm. Math. Phys. \textbf{105}
  (1986), no.~1, 35--47. \MR{847126 (87j:81027)}

\bibitem{dfs}
Piero D'Ancona, Damiano Foschi, and Sigmund Selberg, \emph{Null structure and
  almost optimal local regularity for the {D}irac-{K}lein-{G}ordon system}, J.
  Eur. Math. Soc. (JEMS) \textbf{9} (2007), no.~4, 877--899. \MR{2341835
  (2008k:35388)}

\bibitem{DeFa}
Jean-Marc Delort and Daoyuan Fang, \emph{Almost global existence for solutions
  of semilinear {K}lein-{G}ordon equations with small weakly decaying {C}auchy
  data}, Comm. Partial Differential Equations \textbf{25} (2000), no.~11-12,
  2119--2169. \MR{1789923 (2001g:35165)}

\bibitem{E65}
Robert~E. Edwards, \emph{Functional analysis. {T}heory and applications}, Holt,
  Rinehart and Winston, New York, 1965. \MR{0221256 (36 \#4308)}

\bibitem{ev}
Miguel Escobedo and Luis Vega, \emph{A semilinear {D}irac equation in
  {$H^s({\bf R}^3)$} for {$s>1$}}, SIAM J. Math. Anal. \textbf{28} (1997),
  no.~2, 338--362. \MR{1434039 (97k:35239)}

\bibitem{flr}
R.~Finkelstein, R.~LeLevier, and M.~Ruderman, \emph{Nonlinear spinor fields},
  Phys. Rev. \textbf{83} (1951), no.~2, 326--332.

\bibitem{GV85}
Jean Ginibre and Giorgio Velo, \emph{Time decay of finite energy solutions of
  the nonlinear {K}lein-{G}ordon and {S}chr\"odinger equations}, Ann. Inst. H.
  Poincar\'e Phys. Th\'eor. \textbf{43} (1985), no.~4, 399--442. \MR{824083
  (87g:35208)}

\bibitem{gv}
\bysame, \emph{Smoothing properties and retarded estimates for some dispersive
  evolution equations}, Comm. Math. Phys. \textbf{144} (1992), no.~1, 163--188.
  \MR{1151250 (93a:35065)}

\bibitem{H1}
Lars H\"ormander, \emph{{The analysis of linear partial differential operators.
  I: Distribution theory and Fourier analysis}}, {Classics in Mathematics.
  Reprint of the 2nd edition (1990)}, {Springer}, {Berlin}, 2003.

\bibitem{KaOz}
Jun Kato and Tohru Ozawa, \emph{Endpoint {S}trichartz estimates for the
  {K}lein-{G}ordon equation in two space dimensions and some applications}, J.
  Math. Pures Appl. (9) \textbf{95} (2011), no.~1, 48--71. \MR{2746436
  (2012g:35180)}

\bibitem{KT98}
Markus Keel and Terence Tao, \emph{Endpoint {S}trichartz estimates}, Amer. J.
  Math. \textbf{120} (1998), no.~5, 955--980. \MR{MR1646048 (2000d:35018)}

\bibitem{Kl}
Sergiu Klainerman, \emph{Global existence of small amplitude solutions to
  nonlinear {K}lein-{G}ordon equations in four space-time dimensions}, Comm.
  Pure Appl. Math. \textbf{38} (1985), no.~5, 631--641. \MR{803252 (87e:35080)}

\bibitem{Kl2}
\bysame, \emph{Remark on the asymptotic behavior of the {K}lein-{G}ordon
  equation in {${\bf R}^{n+1}$}}, Comm. Pure Appl. Math. \textbf{46} (1993),
  no.~2, 137--144. \MR{1199196 (93k:35046)}

\bibitem{Ko}
Roman Kosecki, \emph{The unit condition and global existence for a class of
  nonlinear {K}lein-{G}ordon equations}, J. Differential Equations \textbf{100}
  (1992), no.~2, 257--268. \MR{1194810 (93k:35178)}

\bibitem{Kr03}
Joachim Krieger, \emph{Global regularity of wave maps from {$\mathbb{R}^{3+1}$}
  to surfaces}, Comm. Math. Phys. \textbf{238} (2003), 333--366.

\bibitem{Kr04}
\bysame, \emph{Global regularity of wave maps from {$\mathbb{R}^{2+1}$} to
  {$\mathbb{H}^2$}. small energy}, Comm. Math. Phys. \textbf{250} (2004),
  507--580.

\bibitem{Li}
Walter Littman, \emph{Fourier transforms of surface-carried measures and
  differentiability of surface averages}, Bull. Amer. Math. Soc. \textbf{69}
  (1963), 766--770. \MR{0155146 (27 \#5086)}

\bibitem{MNNO}
Shuji Machihara, Makoto Nakamura, Kenji Nakanishi, and Tohru Ozawa,
  \emph{Endpoint {S}trichartz estimates and global solutions for the nonlinear
  {D}irac equation}, J. Funct. Anal. \textbf{219} (2005), no.~1, 1--20.
  \MR{2108356 (2006b:35199)}

\bibitem{MNO}
Shuji Machihara, Kenji Nakanishi, and Tohru Ozawa, \emph{Small global solutions
  and the nonrelativistic limit for the nonlinear {D}irac equation}, Rev. Mat.
  Iberoamericana \textbf{19} (2003), no.~1, 179--194. \MR{1993419
  (2005h:35293)}

\bibitem{MNT10}
Shuji Machihara, Kenji Nakanishi, and Kotaro Tsugawa, \emph{Well-posedness for
  nonlinear {D}irac equations in one dimension}, Kyoto J. Math. \textbf{50}
  (2010), no.~2, 403--451. \MR{2666663 (2011d:35435)}

\bibitem{MSW80}
Bernard Marshall, Walter Strauss, and Stephen Wainger, \emph{{$L^{p}-L^{q}$}
  estimates for the {K}lein-{G}ordon equation}, J. Math. Pures Appl. (9)
  \textbf{59} (1980), no.~4, 417--440. \MR{607048 (82j:35133)}

\bibitem{M88}
Frank Merle, \emph{Existence of stationary states for nonlinear {D}irac
  equations}, J. Differential Equations \textbf{74} (1988), no.~1, 50--68.
  \MR{949625 (89k:81027)}

\bibitem{M-S}
Stephen~J. Montgomery-Smith, \emph{Time decay for the bounded mean oscillation
  of solutions of the {S}chr\"odinger and wave equations}, Duke Math. J.
  \textbf{91} (1998), no.~2, 393--408. \MR{1600602 (99e:35006)}

\bibitem{MS}
Cathleen~S. Morawetz and Walter~A. Strauss, \emph{Decay and scattering of
  solutions of a nonlinear relativistic wave equation}, Comm. Pure Appl. Math.
  \textbf{25} (1972), 1--31. \MR{0303097 (46 \#2239)}

\bibitem{NaSch}
Kenji Nakanishi and Wilhelm Schlag, \emph{Invariant manifolds and dispersive
  {H}amiltonian evolution equations}, Zurich Lectures in Advanced Mathematics,
  European Mathematical Society (EMS), Z\"urich, 2011. \MR{2847755
  (2012m:37120)}

\bibitem{Pe13}
Hartmut Pecher, \emph{Local well-posedness for the nonlinear {D}irac equation
  in two space dimensions}, Commun. Pure Appl. Anal. \textbf{13} (2014), no.~2,
  673--685, Corrigendum in arXiv:1303.1699v6 [math.AP]. \MR{3117368}

\bibitem{S76}
Irving Segal, \emph{Space-time decay for solutions of wave equations}, Advances
  in Math. \textbf{22} (1976), no.~3, 305--311. \MR{0492892 (58 \#11945)}

\bibitem{Sh}
Jalal Shatah, \emph{Normal forms and quadratic nonlinear {K}lein-{G}ordon
  equations}, Comm. Pure Appl. Math. \textbf{38} (1985), no.~5, 685--696.
  \MR{803256 (87b:35160)}

\bibitem{Si}
Thomas~C. Sideris, \emph{Decay estimates for the three-dimensional
  inhomogeneous {K}lein-{G}ordon equation and applications}, Comm. Partial
  Differential Equations \textbf{14} (1989), no.~10, 1421--1455. \MR{1022992
  (90m:35130)}

\bibitem{soler}
Mario Soler, \emph{Classial, stable, nonlinear spinor fields with positive rest
  energy}, Phys. Rev. D \textbf{{1}} ({1970}), no.~{10}, 2766--2769.

\bibitem{St}
Elias~M. Stein, \emph{Harmonic analysis: real-variable methods, orthogonality,
  and oscillatory integrals}, Princeton Mathematical Series, vol.~43, Princeton
  University Press, Princeton, NJ, 1993, With the assistance of Timothy S.
  Murphy, Monographs in Harmonic Analysis, III. \MR{1232192 (95c:42002)}

\bibitem{StTa10}
Jacob Sterbenz and Daniel Tataru, \emph{{Energy dispersed large data wave maps
  in $2+1$ dimensions.}}, {Comm. Math. Phys.} \textbf{298} (2010), no.~1,
  139--230.

\bibitem{sv}
Walter Strauss and Luis V\'azquez, \emph{Stability under dilations of nonlinear
  spinor fields}, Phys. Rev. D \textbf{34} (1986), no.~2, 641--643.

\bibitem{Str77}
Robert~S. Strichartz, \emph{Restrictions of {F}ourier transforms to quadratic
  surfaces and decay of solutions of wave equations}, Duke Math. J. \textbf{44}
  (1977), no.~3, 705--714. \MR{0512086 (58 \#23577)}

\bibitem{Tao2}
Terence Tao, \emph{Spherically averaged endpoint {S}trichartz estimates for the
  two-dimensional {S}chr\"odinger equation}, Comm. Partial Differential
  Equations \textbf{25} (2000), no.~7-8, 1471--1485. \MR{1765155 (2001h:35038)}

\bibitem{tao}
\bysame, \emph{Global regularity of wave maps. {II}. {S}mall energy in two
  dimensions}, Comm. Math. Phys. \textbf{224} (2001), no.~2, 443--544.
  \MR{1869874 (2002h:58052)}

\bibitem{Tao-bil}
\bysame, \emph{A counterexample to an endpoint bilinear {S}trichartz
  inequality}, Electron. J. Differential Equations (2006), No. 151, 6.
  \MR{2276576 (2007h:35043)}

\bibitem{tat}
Daniel Tataru, \emph{On global existence and scattering for the wave maps
  equation}, Amer. J. Math. \textbf{123} (2001), no.~1, 37--77. \MR{1827277
  (2002c:58045)}

\bibitem{vW}
Wolf von Wahl, \emph{{$L^{p}$}-decay rates for homogeneous wave-equations},
  Math. Z. \textbf{120} (1971), 93--106. \MR{0280885 (43 \#6604)}

\end{thebibliography}

\end{document}